\documentclass[hidelinks,12pt]{article}
\usepackage{amsmath, amssymb, amscd, amsthm, amsfonts}
\usepackage{graphicx}
\usepackage[colorlinks=false,citecolor=black,urlcolor=blue,
	hypertexnames=false]{hyperref}
\usepackage{mathtools}
\usepackage[numbers]{natbib}
\usepackage{indentfirst}
\usepackage{bbold}

\def\ifdec{\ifx34}   

\usepackage[inline,final]{showlabels}

\usepackage{tikz}
\usetikzlibrary{bayesnet,positioning,matrix,shapes.geometric}
\usetikzlibrary{arrows.meta}
\usepackage[skins,breakable]{tcolorbox}

\newcommand{\pref}[1]{(\ref{#1})}

\newtcolorbox{mybox}[1][]{enhanced jigsaw,breakable,pad at break=1mm,
  oversize,left=10mm, interior style={color=yellow!20},
  colframe=blue!40,nobeforeafter=,#1}

\usepackage{enumitem}
\usepackage{physics}
\usepackage{tikz}

\newtheorem{theorem}{Theorem}

\newtheorem{lemma}{Lemma}[section]
\newtheorem{proposition}[lemma]{Proposition}
\newtheorem{remark}[theorem]{Remark}

\newcommand{\N}{\mathbb{N}}
\newcommand{\R}{\mathbb{R}}

\newcommand{\C}{\mathbb{C}}

\renewcommand{\P}{{\bf P}}
\newcommand{\E}{{\bf E}}
\newcommand{\X}{\mathcal{X}}
\newcommand{\F}{\mathcal{F}}
\newcommand{\one}{\mathbb{1}}
\def\half{{1 \over 2}}
\def\and{\ \  {\rm and} \ \ }
\renewcommand{\implies}{\ {\Rightarrow} \ }

\newcommand{\tvnorm}[1]{\norm{#1}_{\operatorname{TV}}}

\newcommand{\s}{\mathcal{S}}

\newcommand{\vv}{$V$}
\newcommand{\lv}{L^{\infty}_{V}}
\newcommand{\lvz}{L^{\infty}_{V,0}}

\newcommand{\lppi}{L^{p}(\pi)}
\newcommand{\lpp}{L^{2}(\pi)}
\def\condref#1{(\ref{#1})}
\def\qref#1{Q$\,$\ref{sec-open}.\ref{#1}}
\def\ifextralemmas{\ifx34}

\def\checkpage#1{
\dimen5 = \pagetotal
\ifdim \dimen5 < \pagegoal
\advance \dimen5 by #1
\ifdim \dimen5 > \pagegoal
\newpage
\fi
\fi
}

\begin{document}

\centerline{\LARGE Equivalences of Geometric Ergodicity of Markov Chains}

\bigskip
\centerline{by \ (in alphabetical order)}

\bigskip
\centerline{\large Marco A.\ Gallegos-Herrada, \ David Ledvinka,
	\ and \ Je{f}frey S.\ Rosenthal\footnote
{Corresponding author: jeff@math.toronto.edu
\quad ORCID: 0000-0002-5118-6808}
}

\medskip
\centerline{\sl Departments of Statistics and Mathematics,
University of Toronto, Canada}

\bigskip
\centerline{(Version of: \today.)}

\allowdisplaybreaks

\begin{abstract}\noindent
This paper gathers together different conditions which are all equivalent
to geometric ergodicity of time-homogeneous Markov chains on general state
spaces.  A total of 34 different conditions are presented (27 for general
chains plus 7 for reversible chains), some old and some new, in terms
of such notions as convergence bounds, drift conditions, spectral
properties, etc., with different assumptions about the distance metric
used, finiteness of function moments, initial distribution, uniformity of
bounds, and more.  Proofs of the connections between the different
conditions are provided, somewhat self-contained but using some results from
the literature where appropriate.
\end{abstract}

\section{Introduction}\label{sec-introduction}

The increasing importance of Markov chain Monte Carlo (MCMC) algorithms
(see e.g.~\cite{handbook} and the many references therein) has focused
attention on the rate of convergence of (time-homogeneous) Markov chains
to their stationary
distribution.  While it is most useful to have
explicit quantitative bounds on the
distance to stationarity (see e.g.~\cite{Rosenthal2002,JonesHobert}
and the references
therein), qualitative convergence bounds are often more feasible to obtain.
The most commonly-used qualitative convergence property
is geometric ergodicity, i.e.\ exponentially fast convergence to stationarity,
which has been widely studied (e.g.\
\cite{tierney1994markov,MeynTweedie,probsurv}), and indeed has become a
{\it de facto} method of assessing the value of MCMC algorithms.

In addition to fast convergence, geometric ergodicity also
guarantees a Markov chain Central Limit Theorem (CLT),
i.e.\ the convergence of
scaled sums of functional values
to a fixed normal
distribution, for all functionals
with finite
$2+\delta$ moments \cite[Theorem~18.5.3]{ibragimov}
(see also \cite{hobertregen}),
or even just $2^{\rm nd}$ moments assuming reversibility
\cite{hybrid}.
Such CLTs are helpful for understanding the errors which arise from Monte
Carlo estimation (see e.g.\ \cite{tierney1994markov, RTmet,
jones2004markov}).
However, geometric ergodicity and CLTs do not hold for all
Markov chains nor all MCMC algorithms
(see e.g.\ \cite{negclt} and \cite[Theorem 22]{probsurv}).

For certain types of MCMC algorithms, geometric ergodicity is fairly well
understood.  For example,
it is known that an Independence Sampler is geometrically ergodic if and
only if its proposal density is bounded below by a constant multiple of
the target density \cite{Liu1996},
and that the popular Random-Walk Metropolis algorithm is geometrically
ergodic essentially if and only if its target distribution has
exponentially light tails \cite{mengersen,RTmet}.
However, for many other complicated Markov chains and MCMC algorithms,
geometric ergodicity is not clear.

One promising way of establishing geometric ergodicity
is to show that some other properties of Markov
chains imply it, or are even equivalent to it.
This has been shown, by
\cite{tierney1994markov, MeynTweedie, hybrid, robertsTweedie:2001} and
others, for properties such as
drift conditions, spectral bounds, and more.  However, such
relationships are scattered throughout the literature, are not always
stated in full generality, and are often presented as just one-way
implications.  
In the current work, we present a total of 34 different
conditions which are
equivalent to geometric ergodicity
for Markov chains on general state spaces
(27 for general chains plus 7 just for reversible chains;
some previously known and some new).
We then provide proofs of all of the equivalences (somewhat
self-contained, though using known results where needed);
see Figure~\ref{fig:statements}.

To illustrate the flavour of the various equivalences, consider the
following:

\def\textindent#1{\indent\llap{#1\enspace}\ignorespaces}
\def\hang{\hangindent\parindent}
\def\myitem{\par\hang\textindent}
\def\point{\myitem{$\bullet$}}

\point
The usual definitions of geometric ergodicity state that the
Markov chain's {\it distance to stationarity}
after $n$ iterations is bounded by a
constant times $\rho^n$ for some $\rho<1$.
But what ``distance'' should be used: total variation, or $V$-norm,
or $L^2(\pi)$?  And, how does the ``constant'' depend on the starting
state $X_0=x$?  Must those constants have finite expected value
with respect to $\pi$?  What about finite $j^{\rm th}$ moments?

\point If the initial state $X_0$ is itself chosen from a non-degenerate
{\it initial distribution} probability measure $\mu$, then
will the convergence to stationarity still be geometric,
at least if $\mu$ is, say, in $L^p(\pi)$?

\point Geometric ergodicity is well-known to be implied by {\it drift
conditions} of the form $PV(x) \le \lambda \, V(x) + b \, \one_{S}(x)$
for some function $V:\X\to[1,\infty]$
and $\lambda<1$ and $b<\infty$ and small set $S$.
But are such drift conditions actually {\it equivalent} to
geometric ergodicity?  And, can the drift function $V$ be taken to have
finite stationary mean? finite $j^{\rm th}$ moment?

\point Geometric ergodicity is also related to
the Markov operator $P$ having a {\it spectral gap}.
But as an operator on what space:
$L^\infty_V$?
for what function $V$?
having which finite moments?
And should the ``gap'' be identified by removing the eigenvalue~1 directly,
or by subtracting off $\Pi$, or by restricting
to the zero-mean space $L^\infty_{V,0}$?

\point Geometric ergodicity is implied by the Markov operator
{\it norm} being less than~1.
But for which operator: $P$, or $P^m$ for some $m\in\N$?
Regarded as an operator on $L^\infty_V$ or $L^\infty_{V,0}$?
For what choice of $V$?  Having which finite moments?

\point If the Markov chain is assumed to be {\it reversible}, so that
the operator $P$ is self-adjoint on $L^2(\pi)$, then in
which of
the above conditions can the operator norm be taken to be $L^2(\pi)$?

\noindent
We shall see that the answer to these questions is, essentially, ``all of
the above''.  That is, we shall state many different conditions, which
cover essentially all of the above possibilities, and shall prove
that they are all equivalent.  In our desire to be thorough, we might have
gone a bit overboard listing so many different conditions,
including some which are just minor variations of each other.
However, we believe that additional equivalent conditions can only help:
the equivalences with weaker assumptions are easier to establish,
while the equivalences with stronger assumptions are most useful for
drawing conclusions or analysing further.
We know from bitter experience that it can be very frustrating to
discover a statement about geometric ergodicity which is almost, but not
quite, exactly what we can verify, or exactly what
is needed to finish a particular proof.
This has led us to adopt a ``the more the merrier'' attitude
regarding different but similar conditions.
The reader can, of course, choose to ignore all conditions
which are not germaine to their work.

As mentioned, many of the equivalences presented herein were already
known; see the Remark after Theorem~\ref{thm-equiv} below.
Thus, this paper falls somewhere in between an expository/review paper and
a original research paper, but we hope it is helpful nonetheless.

Basic definitions necessary to understand the conditions, such as total
variance distance, $\lv$ norms, $L^p(\pi)$ spaces, reversibility, etc.,
are presented in Section~\ref{sec-definitions}.
Then, in Section~\ref{sec-statements}, all of the equivalent conditions are
introduced (Theorem~\ref{thm-equiv}).
Sections~\ref{sec-lemmas} through~\ref{sec-reversible}
are then devoted to proving all of the equivalences;
see Figure~\ref{fig:statements} for a visual guide showing
which implications are proved by which of our results.
Our proofs are somewhat
self-contained, but we do use known results in the
literature (especially \cite{MeynTweedie}) where needed.
Finally, we close in Section~\ref{sec-open} with some future
directions and open problems (\qref{q-irred} through \qref{q-simpleerg}).

\section{Definitions and Background}
\label{sec-definitions}

Throughout this paper, $\Phi = \{X_n\}_{n=0}^\infty$ is a discrete-time,
time-homogeneous {\it Markov chain} on a general state space $\X$ equipped
with a $\sigma$-algebra $\F$.  And, $P$ is the corresponding Markov kernel,
so that $P(x,A) = \P[X_n \in A \, | \, X_{n-1}=x]$ for all $x\in\X$ and
$A\in\F$ and $n\in\N$.
The {\it kernel} $P$ acts to the left on (possibly signed) measures,
and to the right on functions, by:
\begin{align*}
    (\mu P) (A) = \int P(x,A) \, \mu(dx),
&\quad \quad (Pf)(x) = \int f(y) \, P(x,dy) .
\end{align*}
The {\it higher-order transitions} are then defined inductively by:
$$
P^{n}(x,A) \ = \ \int_{\X} P(x,dy) \, P^{n-1}(y,A),
\qquad x\in\X, \ A\in\F, \ n\in\N.
$$

We shall assume throughout $P$ has a {\it stationary distribution}, i.e.\
a probability distribution $\pi$ on $(\X,\F)$ which is preserved by $P$
in the sense that $\pi P = \pi$.
We define $\Pi := \one_{\X}\otimes \pi$ by
\[
\Pi(x,A) \ := \
(\one_{\X}\otimes \pi)(x,A) \ = \  \pi(A),
\qquad x\in\X,\ A\in\F ,
\]
so that
$$
    (\mu \, \Pi)(A) \ := \ \big( \mu( \one_{\X}\otimes\pi) \big) (A) \ = \
    \mu(\X) \, \pi(A) \, .
$$
If $\mu$ is a probability measure, then $(\mu \, \Pi)(A) = \pi(A)$,
and $\mu(P^n - \Pi) = \mu P^n - \pi$.
Also, by stationarity of $\pi$, we have
$(P-\Pi)^{n}=P^n - \Pi$
for each $n\in\N$.

We shall assume that our Markov chain is {\it $\phi$-irreducible},
i.e.\ there exists a non-zero $\sigma$-finite measure $\phi$ on $(\X,\F)$
such that for all $x\in \X$ and $A\subseteq \X$ with $\phi(A) > 0 $,
there is $n\in \N$ with $P^{n}(x,A)>0$.
We shall also assume that it is {\it aperiodic}, i.e.\ there
do not exist $d\geq 2$ and disjoint $\X_{1},\ldots,\X_{d}\subseteq
\X$ of positive $\pi$ measure,
such that $P(x,\X_{i+1})=1$ for all $x\in \X_{i}$ ($i=1,\ldots,d-1$) and
$P(x,\X_{1})=1$ for all $x\in \X_{d}$.
It is well-known (e.g.~\cite{MeynTweedie,probsurv}) that these conditions
guarantee that $P^n(x,A) \to \pi(A)$ as $n\to\infty$
(see also \qref{q-irred} and \qref{q-varbound} below).
Geometric ergodicity then corresponds to the property, which may
or may not hold, that this convergence occurs exponentially quickly.

We shall also assume that the state space $(\X,\F)$ is {\it countably
generated}, i.e.\ that there exists $A_1,A_2,\ldots\in\F$ such that
$\F=\sigma(A_1,A_2,\ldots)$, i.e.\ $\F$ is the smallest $\sigma$-algebra
containing all of the $A_i$.
This technical property ensures the existence of small sets
\cite{Doeblin1940,JainJamison, Orey}
and the measurability of certain functions \cite[Appendix]{hybrid}
(see also \qref{q-countgen} below).

A subset $S\in\F$ is called {\it small} if $\pi(S)>0$ and
there is $m>0$ and a non-zero
measure $\nu$ on $(\X,\F)$ such that
$P^{m}(x,A) \ \geq \ \nu(A)$ for all
$x\in S$ and $A\in\F$, i.e.\ if all of the $m$-step transition
probabilities from within $S$ all have some ``overlap''.
This property is very useful for coupling constructions
and for ensuring convergence to stationarity
(see e.g.\ \cite{MeynTweedie, probsurv}).

The {\it total variation distance} between two probability measures
$\mu_{1}$ and $\mu_{2}$ is defined by:
    \[
    \tvnorm{\mu_{1}-\mu_{2}} \ = \
\sup\limits_{A\in\F}|\mu_{1}(A)-\mu_{2}(A)|
\ \equiv \ \frac{1}{2} \, \sup\limits_{|f|\leq 1 }
\Big| \int f d\mu_1 - \int fd\mu_2 \Big|
    \]
(see e.g.\ \cite[Proposition 3(b)]{probsurv}).
Given a positive function $V:\X \to \R$, we define
\cite[p.~390]{MeynTweedie} the {\it $V$-norm}
$|f|_{V}=\sup\limits_{x\in\X}\frac{|f(x)|}{V(x)}$.  We let $\lv$ be the
vector space of all functions $f:\X\to\R$ such that $|f|_{V} < \infty$,
and let $\lvz = \{ f\in \lv : \pi(f) = 0 \}$.
Then, we define the {\it $V$-norm} of a Markov kernel $P$ as
\begin{align*}
\norm{P}_{\lv} = \sup\limits_{\substack{f\in\lv \\ |f|_V = 1}} |Pf|_{V}; 
\quad \quad \norm{P}_{\lvz} =\sup\limits_{\substack{f\in\lvz\\ |f|_V=1}} |Pf|_{V}.    
\end{align*}
For a (possibly signed) measure $\mu$,
we define $\norm{\mu}_{L^p(\pi)}$ for $1 \le p < \infty$ by
\[
    \norm{\mu}_{L^p(\pi)}^{p}
\ = \
\begin{cases}
    \mu^{+}(\X) + \mu^{-}(\X), &\mathrm{if}\ p=1\\
    {\displaystyle \int_{\X}
\Big| \dfrac{d\mu}{d\pi} \Big|^p d\pi},
			&\mathrm{if}\ \mu \ll \pi\\
    \infty, & \mathrm{otherwise}.
    \end{cases}
\]
(If $p=1$ and $\mu \ll \pi$, then the two definitions coincide.)
We let $L^{p}(\pi)$
be the collection of all signed measures $\mu$ on $(\X,\F)$
with $\norm{\mu}_{L^p(\pi)} < \infty$,
and define the $L^p(\pi)$-norm of a
transition kernel $P$ acting on the set $L^{p}(\pi)$ by:
$$
\norm{P}_{L^{p}(\pi)}
\ = \ \sup\limits_{\norm{\mu}_{L^p(\pi)}=1} \norm{\mu P(\cdot)}_{L^p(\pi)}.
$$
(Note in particular that the $L^p(\pi)$ are collections of signed
{\it measures}, while $L_V^\infty$ and $L_{V,0}^\infty$ are collections
of {\it functions}.)

The transition kernel $P$ is {\it reversible} with respect to $\pi$ if
$
\pi(dx) \, P(x,dy) \ = \ \pi(dy) \, P(y,dx)
$
for all $x,y\in\X$.
This is equivalent to $P$ being a self-adjoint operator
on the Hilbert space $\lpp$, with inner product given by 
\[
\langle \mu, \nu \rangle
\ = \ \int_{\X} \dfrac{d\mu}{d\pi} \, \dfrac{d\nu}{d\pi} \, d\pi.
\]
In particular, $\langle \mu, \pi \rangle = \int_\X {d\mu \over d\pi} \, 1
\, d\pi = \mu(\X)$.
We also let $\pi^{\bot} \coloneqq \{ \mu \in \lpp : \mu(\X) = 0 \}$
be the set of signed measures in $\lpp$ which are ``perpendicular''
to $\pi$, i.e.\ for which $\langle \mu, \pi \rangle \equiv \mu(\X) = 0$.
Our conditions \condref{eq-l2ge-Cmu} through \condref{eq-srl2-p-0} are
only proven to be equivalent for reversible chains (though see
\qref{q-reversible} below).

\def\vecsp{{\cal V}}

Finally, given an operator $P$ on a Banach space (i.e.\ a complete normed
vector space) $\vecsp$,
e.g.\ $\vecsp=\lv$ or $\lpp$, the {\it spectrum} of $P$, denoted by
$\s(P)$ or $\s_\vecsp(P)$,
is the set of all complex numbers $\lambda$ such that $\lambda I -
P$ is not invertible
(see e.g.\ \cite[p.~253]{rudin:1991functional}).
And, the {\it spectral radius} of $P$ is the number
$r(P) = r_\vecsp(P) = \sup\limits_{\lambda\in \s_\vecsp(P)} |\lambda|$.

\section{Main Result: Statement of Equivalences}\label{sec-statements}

We now provide a list of 27 conditions which are always equivalent to
geometric ergodicity of Markov chains, and an additional 7 (for 34 total)
which are also equivalent for reversible chains.  Some of the conditions
are very similar to each other, but are included to allow for maximum
flexibility when establishing or using geometric ergodicity in both
theoretical investigations and applications.  For ease of
comprehension, similar conditions are grouped together under common
subheadings.

\begin{theorem}\label{thm-equiv}
    Let $P$ be the transition kernel of a $\phi$-irreducible, aperiodic Markov chain $\Phi = \{X_n\}$ with stationary probability
distribution $\pi$ on a countably generated measurable state space $(\X,\F)$.
Then the following are equivalent (and all correspond to being
``geometrically ergodic''):

\def\condhead#1{\medskip\checkpage{2cm}$\big.$\hskip-1cm\underbar{\rm #1:}}

\medskip\checkpage{5cm}\noindent\underbar{\rm Geometric Convergence in TV:}\par

    \begin{enumerate}[label=\textit{\roman*}), ref=\textit{\roman*}]

        \item\label{eq-geometric-erg}
       $\Phi$ is geometrically ergodic starting from $\pi$-a.e.\ $x \in \X$
            with constant geometric rate.
This means there is fixed $\rho<1$ such that for
            $\pi$-a.e.\ $x \in \X$ there is $C_x <\infty$ with
            \begin{align*}
                \tvnorm{P^n(x,\cdot) - \pi(\cdot)} \ \leq \ C_x \, \rho^n
\qquad \hbox{\rm for all} \ n\in\N .
            \end{align*}

        \item\label{eq-weak-geometric-erg}
            There exists $A\in\F$ with $\pi(A) > 0$ such that $\Phi$ is 
geometrically ergodic starting from each $x\in A$.
	This means for each $x \in A$, there are $\rho_x <1$
            and $C_x < \infty$ with
            \begin{align*}
                \tvnorm{P^n(x,\cdot) - \pi(\cdot)} \ \leq \ C_x \, \rho_x^n
\qquad \hbox{\rm for all} \ n\in\N .
            \end{align*}

        \item\label{eq-gefslp}
	    There exists $p \in (1,\infty)$ such that $\Phi$ is
geometrically ergodic starting from all probability measures in
$L^p(\pi)$.  This means there is some $p\in (1,\infty)$ such that
for each probability measure $\mu \in L^p(\pi)$
there are constants $\rho_{\mu} < 1$ and $C_{\mu} < \infty$ with
            \begin{align*}
	\tvnorm{\mu P^n(\cdot) - \pi(\cdot)} \ \leq \ C_{\mu} \, \rho_{\mu}^n
\qquad \hbox{\rm for all} \ n\in\N .
            \end{align*}

        \item\label{eq-gefalp}
            For all $p \in (1,\infty)$, $\Phi$ is geometrically ergodic starting from
            all probability measures in $L^p(\pi)$ with geometric
            rate depending only on $p$. This means for each
$p \in (1,\infty)$, there is $\rho_p < 1$ such that
for each probability measure $\mu \in L^p(\pi)$
there is $C_{p,\mu} < \infty$ with
            \begin{align*}
	\tvnorm{\mu P^n(\cdot) - \pi(\cdot)} \ \leq \ C_{p,\mu} \, \rho_p^n
\qquad \hbox{\rm for all} \ n\in\N .
            \end{align*}

        \item\label{eq-ssgesup}
            There exists a small set $S \in \F$ such that $\Phi$ is
            geometrically ergodic uniformly over starting states
            within $S$. This means there are constants
            $\rho_S < 1$ and $C_S < \infty$ with
            \begin{align*}
	    \sup_{x\in S} \
	\tvnorm{P^n(x,\cdot) - \pi(\cdot)} \ \leq \ C_{S} \, \rho_{S}^n
\qquad \hbox{\rm for all} \ n\in\N .
            \end{align*}

        \item\label{eq-ssge}
            There exists a small set $S \in \F$ such that $\Phi$ is
            geometrically ergodic starting from the stationary distribution
            restricted to $S$. This means there are constants
            $\rho_S < 1$ and $C_S < \infty$ with
            \begin{align*}
	\tvnorm{\pi_{S} P^n(\cdot) - \pi(\cdot)} \ \leq \ C_{S} \, \rho_{S}^n
\qquad \hbox{\rm for all} \ n\in\N ,
            \end{align*}
            where $\pi_{S}$ is the probability measure defined by
                $\pi_{S}(A) \, = \, \pi(S \cap A) \bigm/ \pi(S)$ for $A\in\F$.

\end{enumerate}

\condhead{Geometric Return Time}

    \begin{enumerate}[label=\roman*), ref=\textit{\roman*}, resume]

        \item\label{eq-tau-c}
	There exists a small set $S\in\F$ and constant $\kappa > 1$ such that
            \begin{align*}
	\sup\limits_{x\in S} \, \E_{x}[\kappa^{\tau_{S}}] \ < \ \infty
            \end{align*}
where $\tau_S$ is the first return time to $S$, and $\E_x$ is expected
value conditional on $X_0=x$.

\condhead{$V$-Function Drift Condition}

        \item\label{eq-dc-nomom}
	    There exists a $\pi$-a.e.-finite measurable function $V: \X
\to [1,\infty]$, a small set $S \in \F$, and constants $\lambda < 1$
and $b < \infty$ with
            \begin{align*}
                PV(x) \ \leq \ \lambda \, V(x) + b \, \one_{S}(x)
\qquad \hbox{\rm for all} \ x\in\X .
            \end{align*}

        \item\label{eq-dc-allj}
            For all $j\in \N$, there exists a $\pi$-a.e.-finite measurable function $V : \X \to [1,\infty]$, a small set $S \in \F$, and constants
$\lambda<1$ and $b<\infty$ with $\pi(V^{j})<\infty$ and
            \begin{align*}
                PV(x) \ \leq \ \lambda \, V(x) + b \, \one_{S}(x)
\qquad \hbox{\rm for all} \ x\in\X .
            \end{align*}

\condhead{$V$-Uniform Convergence}

        \item\label{eq-vuex-nomom}
            There exists a $\pi$-a.e.-finite measurable function
            $V : \X \to [1,\infty]$ such that $\Phi$ is $V$-uniformly ergodic.
This means there is $\rho<1$ and $C<\infty$ such that
            \begin{align*}
\sup_{|f| \leq V} \, \big| P^n f(x) - \pi(f) \big|
	\ \leq \ C \, V(x) \, \rho^n
\qquad \hbox{\rm for all} \ x\in\X \ {\rm and} \ n\in\N.
            \end{align*}

        \item\label{eq-vuex-allj}
	For all $j\in \N$, there exists a $\pi$-a.e.-finite measurable function
            $V : \X \to [1,\infty]$ with
            $\pi(V^{j}) <\infty$, such that $\Phi$
            is $V$-uniformly ergodic.
This means there is $\rho<1$ and $C<\infty$ such that
            \begin{align*}
\sup_{|f| \leq V} \, \big| P^n f(x) - \pi(f) \big|
	\ \leq \ C \, V(x) \, \rho^n
\qquad \hbox{\rm for all} \ x\in\X \ {\rm and} \ n\in\N.
            \end{align*}

        \item\label{eq-vuemu-nomom}
            There exists a $\pi$-a.e.-finite measurable function
            $V : \X \to [1,\infty]$, and constants $\rho<1$ and $C<\infty$,
such that for each probability measure $\mu$ on $\X$ with $\mu(V)<\infty$,
            \begin{align*}
\sup_{|f| \leq V} \big|\mu P^n(f) - \pi(f)\big|
\ \leq \ C \, \mu(V) \, \rho^n
\qquad \hbox{\rm for all} \ n\in\N.
            \end{align*}

        \item\label{eq-vuemu-allj}
    For all $j\in\N$, there exists a $\pi$-a.e.-finite measurable function
            $V : \X \to [1,\infty]$ with $\pi(V^j)<\infty$,
and constants $\rho<1$ and $C<\infty$,
such that for each probability measure $\mu$ on $\X$ with $\mu(V)<\infty$,
            \begin{align*}
\sup_{|f| \leq V} \big|\mu P^n(f) - \pi(f)\big|
\ \leq \ C \, \mu(V) \, \rho^n
\qquad \hbox{\rm for all} \ n\in\N.
            \end{align*}

\end{enumerate}

\condhead{Spectral Gap}

    \begin{enumerate}[label=\roman*), ref=\textit{\roman*}, resume]

        \item\label{eq-sg-inf-somej}
There exists $j\in\N$ and
a $\pi$-a.e.-finite measurable function \vv$: \X \to
            [1,\infty] $ with $\pi(V^j)<\infty$,
such that $P$ has
a spectral gap as an operator on $\lv$,
meaning 1 is an eigenvalue of $P$
(which must have multiplicity 1 by Lemma~\ref{lem-P-Linf}),
and there is $\rho<1$ such that
            \[
            \s_{\lv}(P) \setminus \{1\} \ \subseteq \
\{ z\in \mathbb{C} : |z| \leq \rho \}.
            \]

        \item\label{eq-sg-inf-allj}
For all $j\in\N$, there exists
a $\pi$-a.e.-finite measurable function \vv$: \X \to
            [1,\infty] $ with $\pi(V^j)<\infty$,
such that $P$ has
a spectral gap as an operator on $\lv$,
meaning 1 is an eigenvalue of $P$
(which must have multiplicity 1 by Lemma~\ref{lem-P-Linf}),
and there is $\rho<1$ such that
            \[
            \s_{\lv}(P) \setminus \{1\} \ \subseteq \
\{ z\in \mathbb{C} : |z| \leq \rho \}.
            \]

\condhead{Spectral Radius}

        \item\label{eq-srlinf-p-pi-somej}
            There exists $j\in\N$ and
a $\pi$-a.e.-finite measurable function $V:\X \to
            [1,\infty]$ with $\pi(V^j)<\infty$, such that
$P-\Pi$ has spectral radius less than one as an operator
on $L^{\infty}_{V}$, i.e.\
            \[
            r_{L^{\infty}_{V}}(P-\Pi)< 1.
            \]

        \item\label{eq-srlinf-p-pi-allj}
            For all $j\in\N$, there exists a $\pi$-a.e.-finite measurable
function $V : \X \to [1,\infty]$ with $\pi(V^{j})<\infty$, such that
$P-\Pi$ has spectral radius less than one as an operator
on $L^{\infty}_{V}$, i.e.\
            \[
            r_{L^{\infty}_{V}}(P-\Pi)< 1.
            \]

        \item\label{eq-srlinf-p-0-somej}
            There exists $j\in\N$ and
a $\pi$-a.e.-finite 
            measurable function $V : \X \to [1,\infty]$
with $\pi(V^j)<\infty$, such that 
$P$ has spectral radius less than one as an operator
on $L^{\infty}_{V,0}$, i.e.\
            \[
            r_{L^{\infty}_{V,0}}(P)< 1.
            \]

        \item\label{eq-srlinf-p-0-allj}
For all $j\in\N$,
            there exists a $\pi$-a.e.-finite 
            measurable function $V : \X \to [1,\infty]$
with $\pi(V^j)<\infty$, such that 
$P$ has spectral radius less than one as an operator
on $L^{\infty}_{V,0}$, i.e.\
            \[
            r_{L^{\infty}_{V,0}}(P)< 1.
            \]

\condhead{$L^\infty_V$ Operator Norm}

        \item\label{eq-nlinf-p-pi-somej}
            There exists $j,m\in\N$ and a $\pi$-a.e.-finite measurable
            function $V : \X \to [1,\infty]$ with 
            $\pi (V^j)<\infty$, such that
            \begin{align*}
                \norm{P^{m} - \Pi}_{L^{\infty}_{V}} < 1.
            \end{align*}

        \item\label{eq-nlinf-p-pi-allj}
            For all $j\in \N$, there exists $m\in\N$ and
a $\pi$-a.e.-finite measurable
            function $V : \X \to [1,\infty]$ such that $\pi (V^j)<\infty$ and
            \begin{align*}
                \norm{P^{m} - \Pi}_{L^{\infty}_{V}} < 1.
            \end{align*}

        \item\label{eq-voinf-somej}
            There exists $j,m\in\N$ and
a $\pi$-a.e.-finite measurable function $V : \X \to [1,\infty]$
with $\pi(V^j)< \infty$, such that
            \begin{align*}
                \norm{P^{m}}_{L^{\infty}_{V,0}} < 1.
            \end{align*}

        \item\label{eq-voinf-allj}
            For all $j\in \N$, there exists $m\in\N$ and
a $\pi$-a.e.-finite measurable function $V : \X \to [1,\infty]$
with $\pi(V^j)< \infty$, such that
            \begin{align*}
                \norm{P^{m}}_{L^{\infty}_{V,0}} < 1.
            \end{align*}

        \item\label{eq-vinf-somej}
            There exists $j\in\N$ and
a $\pi$-a.e.-finite measurable function
$V : \X \to [1,\infty]$ with $\pi(V^j)<\infty$, and constants $\rho<1$ and
$C<\infty$, such that
            \begin{align*}
                \norm{P^n - \Pi}_{L^{\infty}_{V}} \ \leq \ C \, \rho^n
\qquad \hbox{\rm for all} \ n\in\N.
            \end{align*}

        \item\label{eq-vinf-allj}
            For all $j\in \N$, there exists a $\pi$-a.e.-finite 
            measurable function $V : \X \to [1,\infty]$
with $\pi(V^j)<\infty$, and constants $\rho<1$ and $C<\infty$, such that
            \begin{align*}
                \norm{P^n - \Pi}_{L^{\infty}_{V}} \ \leq \ C \, \rho^n
\qquad \hbox{\rm for all} \ n\in\N.
            \end{align*}

        \item\label{eq-vinf0-somej}
            There exists $j\in\N$ and
a $\pi$-a.e.-finite measurable function
            $V : \X \to [1,\infty]$ with $\pi(V^j)<\infty$,
and constants $\rho<1$ and $C<\infty$, such that
            \begin{align*}
                \norm{P^n}_{L^{\infty}_{V,0}} \ \leq \ C \, \rho^n
\qquad \hbox{\rm for all} \ n\in\N.
            \end{align*}

        \item\label{eq-vinf0-allj}
            For all $j\in \N$, there exists a $\pi$-a.e.-finite 
            measurable function $V : \X \to [1,\infty]$ with $\pi(V^j)<\infty$,
and constants $\rho<1$ and $C<\infty$, such that
            \begin{align*}
                \norm{P^n}_{L^{\infty}_{V,0}} \ \leq \ C \, \rho^n
\qquad \hbox{\rm for all} \ n\in\N.
            \end{align*}

    \end{enumerate}

\condhead{Conditions Assuming Reversibility}

\medskip\noindent Furthermore, if $\Phi$ is reversible, then
the following are also equivalent to the above:

    \begin{enumerate}[label=\roman*), ref=\textit{\roman*}, resume]

        \item\label{eq-l2ge-Cmu} 
            $\Phi$ is $L^2(\pi)$-geometrically ergodic starting from any
            probability measure $\mu$ in $L^2(\pi)$ with uniform convergence
            rate.
This means there is $\rho<1$ such that
for each probability measure $\mu \in L^2(\pi)$,
there is a constant $C_{\mu} < \infty$ such that
            \begin{align*}
\norm{\mu P^n(\cdot) - \pi(\cdot)}_{L^2(\pi)} \ \leq \ C_{\mu} \, \rho^n
\qquad \hbox{\rm for all} \ n\in\N.
            \end{align*}

        \item\label{eq-l2ge-noC}
	There exists $\rho<1$ such that for each probability measure $\mu \in
            L^2(\pi)$,
            \begin{align*}
\norm{\mu P^n(\cdot) - \pi(\cdot)}_{L^2(\pi)} \ \leq \
\norm{\mu - \pi}_{\lpp} \, \rho^n
\qquad \hbox{\rm for all} \ n\in\N.
            \end{align*}

        \item\label{eq-sg-l2}
            $P$ has a spectral gap
as an operator on $L^2(\pi)$, meaning that
1 is an eigenvalue of $P$
(which must have multiplicity 1 by Lemma~\ref{lem-P-Linf}),
and there is $\rho<1$ with
            \[
    \s_{L^2(\pi)}(P) \setminus \{1\} \ \subseteq \ \{z \in \C: |z| \leq \rho\}.
            \]

        \item\label{eq-srl2-p-pi}
$P-\Pi$ has spectral radius less than one as an operator on $L^2(\pi)$, i.e.\
            \[
            r_{L^2(\pi)}(P - \Pi) \ < \ 1.
            \]

        \item\label{eq-nl2-p-pi}
$P-\Pi$ has operator norm less than one as an operator on $L^2(\pi)$, i.e.\
            \begin{align*}
                \norm{P - \Pi}_{L^2(\pi)} \ < \ 1.
            \end{align*}

        \item\label{eq-nl2-p-0}
$P$ has operator norm less than one as an operator on $\pi^\bot$,
i.e.\
            \begin{align*}
                \norm{P}_{\pi^\bot} \ < \ 1.
            \end{align*}
            
        \item\label{eq-srl2-p-0}
$P|_{\pi^\bot}$ has spectral radius
less than one as an operator on $\pi^\bot$, i.e.\
            \begin{align*}
                r_{\pi^\bot}(P) \ < \ 1.
			\end{align*}

    \end{enumerate}

\end{theorem}

\ifx34
\begin{remark}
Since we always have $V(x) \le V(x)^j + 1$, the phrase
``$\pi(V)<\infty$'' is equivalent to the phrase
``$\pi(V^j)<\infty$ for some $j\in\N$'' in conditions
\condref{eq-dc-one}, \condref{eq-sg-inf-one},
\condref{eq-srlinf-p-pi-one}, \condref{eq-srlinf-p-0-one},
\condref{eq-nlinf-p-pi-one}, \condref{eq-voinf-one},
\condref{eq-vinf-one}, and \condref{eq-vinf0-one} above.
\end{remark}
\fi

\medskip\noindent\bf Remark. \rm
A number of the above equivalences are already known, as follows.
The fact that \condref{eq-ssge} implies \condref{eq-geometric-erg} 
was shown in \cite{vere-jones1962}
on countable state spaces, and then in
\cite[Theorem~1]{nummelinTweedie:1978} on general state spaces.
The equivalence of
\condref{eq-ssge},
\condref{eq-tau-c},
and \condref{eq-dc-nomom},
together with the fact that they imply
\condref{eq-geometric-erg},
was presented in \cite[Theorem~15.0.1]{MeynTweedie}.
The equivalence of
\condref{eq-dc-nomom},
\condref{eq-vuex-nomom},
\condref{eq-nlinf-p-pi-somej},
and \condref{eq-vinf0-somej}
was presented in \cite[Theorem~16.0.1]{MeynTweedie}.
The equivalence of the group
\condref{eq-geometric-erg},
\condref{eq-ssge},
\condref{eq-vuex-nomom},
\condref{eq-vuex-allj},
\condref{eq-nlinf-p-pi-allj},
and \condref{eq-voinf-allj}
was presented in \cite[Proposition~1]{hybrid},
and the equivalence (assuming reversibility) of the group
\condref{eq-l2ge-Cmu},
\condref{eq-l2ge-noC},
and \condref{eq-nl2-p-0}
was presented in \cite[Theorem~2]{hybrid},
together with the fact that the first group implies the second.
The reverse implication, that the second group implies the first,
was then shown in \cite{robertsTweedie:2001}.
Discussions related to the spectral gap conditions
\pref{eq-sg-inf-somej} and \pref{eq-sg-inf-allj} and \pref{eq-sg-l2}
appear in \cite{KM1}.
The equivalence of
\pref{eq-sg-inf-somej} and \pref{eq-dc-nomom}
is shown in \cite[Proposition~1.1]{KM2},
and the equivalence of \pref{eq-srl2-p-pi} and \pref{eq-geometric-erg}
for reversible chains
is shown in \cite[Proposition~1.2]{KM2}.
Our Theorem~\ref{thm-equiv} is an attempt to combine and
bring together all of these various results, and add others too.
(Since initiating this work, we also learned of the recent review
\cite{bradley:2019exposition}, which presents certain equivalences
for reversible chains in terms of mixing conditions and
maximal correlations, which complement some of our conditions
\condref{eq-l2ge-Cmu} through \condref{eq-srl2-p-0}.
In addition, the recent volume \cite{Douc} expands upon
much of the material in \cite{MeynTweedie}.)
\bigskip

Most of the remainder of this paper is devoted to proving
Theorem~\ref{thm-equiv}.  The proof is divided up into different
sections below, in terms of which types of conditions are being
considered:
Section~\ref{sec-lemmas} provides some preliminary lemmas,
Section~\ref{sec-geometric} relates to various ``Geometric'' conditions,
Section~\ref{sec-V} relates to various conditions involving
$V$ functions and $L^\infty_V$ bounds,
Section~\ref{sec-spectral} relates to various spectral conditions,
and Section~\ref{sec-reversible} relates to various conditions
for reversible chains.
To help the reader (and ourselves) keep track, Figure~\ref{fig:statements}
provides a diagram showing which of our results prove implications
between which of the equivalent conditions.  Our proofs are somewhat
self-contained, but we use known results from the literature
(especially \cite{MeynTweedie}) where appropriate.
Section~\ref{sec-open} then presents some future
directions and open problems.

\ifx34
aside from a few places (notably
Lemmas~\ref{lem-small_sets_exist}
and~\ref{lem-total_variation_is_measurable}
and Propositions~\ref{proof:v->i}
and~\ref{proof:iii->xviii}
and~\ref{proof:v->xxiv}
and~\ref{proof:xxiv->vi})
\fi

\begin{figure}
    \centering
\begin{tikzpicture}[thick,scale=0.5, every node/.style={transform shape}, node distance = 5cm]
\begin{scope}[
every node/.style={fill=red!20, circle, thick, draw, minimum size=1.0cm,
  scale=1.0, regular polygon, regular polygon sides=6,
  inner sep=1.0, outer sep=0.5},
V-unif/.style={fill=blue!15,circle},
spec/.style={fill=brown!20, regular polygon, regular polygon sides=5,
			minimum size=1.0cm, inner sep=-1.5},
Linf/.style={fill=purple!30, regular polygon, regular polygon sides=9,
	scale=1.0, minimum size=1.0cm, inner sep=-2},
L2/.style={fill=green!15,rectangle, minimum size=1.0cm, inner sep=1.0}
]
    \node (n-geometric-erg) at (14,-15) {$\ref{eq-geometric-erg}$};
    \node (n-weak-geometric-erg) at (20,-13) {$\ref{eq-weak-geometric-erg}$};
    \node (n-gefalp) at (10,0) {$\ref{eq-gefalp}$};
    \node (n-gefslp) at (10,-5) {$\ref{eq-gefslp}$};
    \node (n-ssgesup) at (21,-8) {$\ref{eq-ssgesup}$};
    \node (n-ssge) at (17,-8) {$\ref{eq-ssge}$};
    \node (n-tau-c) at (12.5,-10) {$\ref{eq-tau-c}$};
    \node[V-unif] (n-dc-allj) at (8,-7) {$\ref{eq-dc-allj}$};
    \node[V-unif] (n-dc-nomom) at (8,-12) {$\ref{eq-dc-nomom}$};
    \node[V-unif] (n-vuex-nomom) at (3,-9) {$\ref{eq-vuex-nomom}$};
    \node[V-unif] (n-vuex-allj) at (8,12) {$\ref{eq-vuex-allj}$};
    \node[V-unif] (n-vuemu-nomom) at (3,-15) {$\ref{eq-vuemu-nomom}$};
    \node[V-unif] (n-vuemu-allj) at (10,5) {$\ref{eq-vuemu-allj}$};
    \node[spec] (n-sg-inf-somej) at (-6,-5.5) {$\ref{eq-sg-inf-somej}$};
    \node[spec] (n-sg-inf-allj) at (-6,10) {$\ref{eq-sg-inf-allj}$};
  \node[spec] (n-srlinf-p-pi-somej) at (-2,-9.5) {$\ref{eq-srlinf-p-pi-somej}$};
  \node[spec] (n-srlinf-p-pi-allj) at (-3,12) {$\ref{eq-srlinf-p-pi-allj}$};
    \node[spec] (n-srlinf-p-0-somej) at (-6,0) {$\ref{eq-srlinf-p-0-somej}$};
    \node[spec] (n-srlinf-p-0-allj) at (-6,4.5) {$\ref{eq-srlinf-p-0-allj}$};
    \node[Linf] (n-nlinf-p-pi-somej) at (-1,-4) {$\ref{eq-nlinf-p-pi-somej}$};
    \node[Linf] (n-nlinf-p-pi-allj) at (3,12) {$\ref{eq-nlinf-p-pi-allj}$};
    \node[Linf] (n-voinf-somej) at (-2,2) {$\ref{eq-voinf-somej}$};
    \node[Linf] (n-voinf-allj) at (-1,7) {$\ref{eq-voinf-allj}$};
    \node[Linf] (n-vinf-somej) at (3,2) {$\ref{eq-vinf-somej}$} ;
    \node[Linf] (n-vinf-allj) at (4,6) {$\ref{eq-vinf-allj}$} ;
    \node[Linf] (n-vinf0-somej) at (6,-2.5) {$\ref{eq-vinf0-somej}$};
    \node[Linf] (n-vinf0-allj) at (6,2) {$\ref{eq-vinf0-allj}$};
    \node[L2] (n-l2ge-Cmu) at (15,1) {$\ref{eq-l2ge-Cmu}$};
    \node[L2] (n-l2ge-noC) at (15,-4) {$\ref{eq-l2ge-noC}$};
    \node[L2] (n-sg-l2) at (20,8) {$\ref{eq-sg-l2}$};
    \node[L2] (n-srl2-p-pi) at (20,-2) {$\ref{eq-srl2-p-pi}$};
    \node[L2] (n-nl2-p-pi) at (20,3) {$\ref{eq-nl2-p-pi}$};
    \node[L2] (n-nl2-p-0) at (15,5) {$\ref{eq-nl2-p-0}$};
    \node[L2] (n-srl2-p-0) at (15,10) {$\ref{eq-srl2-p-0}$};
\end{scope}

\begin{scope}[>={Stealth[black]},
              every node/.style={fill=yellow!50,circle,scale=0.9,pos=.5,
				inner sep=2, minimum size=25},
              every edge/.style={draw=red,very thick}]
    \path [<->] (n-voinf-somej) edge node {$\ref{proof:xiii->xii}$} (n-nlinf-p-pi-somej);
    \path [<->] (n-voinf-allj) edge node {$\ref{proof:xiii->xii}$} (n-nlinf-p-pi-allj);
    \path [->] (n-vuex-allj) edge node[pos=.5] {$\ref{proof:viii->xiv}$} (n-vinf-allj);
\path [<->] (n-nlinf-p-pi-somej) edge[bend left=0] node[pos=.5] {$\ref{proof:x<->xii}$} (n-srlinf-p-pi-somej);
\path [<->] (n-nlinf-p-pi-allj) edge[bend left=0] node[pos=.5] {$\ref{proof:x<->xii}$} (n-srlinf-p-pi-allj);
    \path [->] (n-vuex-nomom) edge node[pos=.4] {$\ref{proof:viii->vii}$} (n-vuemu-nomom); 
    \path [->] (n-vuex-allj) edge[bend left=20] node[pos=.4] {$\ref{proof:viii->vii}$} (n-vuemu-allj); 
    \path [->] (n-vuemu-allj) edge node[pos=.4] {$\ref{proof:vii->iii}$} (n-gefalp); 
    \path [->] (n-vuemu-nomom) edge node[pos=.5] {$\ref{proof:mutoge}$} (n-geometric-erg); 

    \path [->] (n-nl2-p-pi) edge node[pos=.4] {$\ref{proof:revnewa}$} (n-l2ge-noC); 
    \path [->] (n-gefalp) edge node[pos=.4] {$\ref{proof:iii->iv}$} (n-gefslp); 
    \path [<->] (n-srl2-p-0) edge[bend left=60] node[pos=.5] {$\ref{proof:xviii<->xix}$} (n-sg-l2);
    \path [->] (n-nl2-p-0) edge[bend left=40] node[pos=.5] {$\ref{proof:xx<->xxi}$} (n-nl2-p-pi);
    \path [<->] (n-nl2-p-pi) edge[bend left=60] node[pos=.5] {$\ref{proof:xix<->xx}$} (n-srl2-p-pi);
    \path [<->] (n-srl2-p-0) edge[bend right=60] node[pos=.5] {$\ref{proof:xxi<->xxii}$} (n-nl2-p-0);
    \path [->] (n-gefslp) edge node {$\ref{proof:iv->v}$} (n-ssge);
    \path [->] (n-ssge) edge node {$\ref{proof:v->i}$} (n-geometric-erg);
    \path [->] (n-geometric-erg) edge node {$\ref{proof:i->ii}$} (n-weak-geometric-erg);
    \path [->] (n-weak-geometric-erg) edge node[pos=.45] {$\ref{proof:ii->vsup}$} (n-ssgesup);
    \path [->] (n-ssgesup) edge[bend right=60] node {$\ref{proof:ii->v}$} (n-ssge);
    \path [<->] (n-srlinf-p-0-somej) edge[bend left=40] node[pos=.5] {$\ref{proof:xi<->xiii}$} (n-voinf-somej);
    \path [<->] (n-srlinf-p-0-allj) edge[bend right=40] node[pos=.5] {$\ref{proof:xi<->xiii}$} (n-voinf-allj);
    \path [<->] (n-vinf0-somej) edge[bend left=00] node[pos=.5] {$\ref{proof:xiv<->xv}$} (n-vinf-somej);
    \path [<->] (n-vinf0-allj) edge[bend right=40] node[pos=.5] {$\ref{proof:xiv<->xv}$} (n-vinf-allj);
    \path [->] (n-vinf-allj) edge[bend right=40] node[pos=.4] {$\ref{proof:choosejone}$} (n-vinf-somej);
    \path [->] (n-ssge) edge node {$\ref{proof:v->xxiv}$} (n-tau-c);
    \path [->] (n-tau-c) edge node[pos=.4] {$\ref{proof:xxiv->vi}$} (n-dc-nomom);
    \path [<-] (n-dc-allj) edge[bend left=0] node[pos=.6] {$\ref{proof:vi'<->vi}$} (n-dc-nomom);
\path [<->] (n-nlinf-p-pi-somej) edge node {$\ref{proof:6.7a}$} (n-vinf-somej);
\path [<->] (n-nlinf-p-pi-allj) edge node {$\ref{proof:6.7a}$} (n-vinf-allj);
\path [->] (n-vinf-somej) edge node {$\ref{proof:6.7b}$} (n-vuex-nomom);
    \path [->] (n-gefalp) edge node {$\ref{proof:iii->xviii}$} (n-nl2-p-0);
    \path [->] (n-l2ge-Cmu) edge node {$\ref{proof:xvi->xvii'}$} (n-gefslp);
    \path [->] (n-l2ge-noC) edge node[pos=.42] {$\ref{proof:xvii->xvi}$} (n-l2ge-Cmu);

\path[->] (n-dc-allj) edge node[pos=.8] {\ref{proof:MT15}} (n-vuex-allj);
\path[<->] (n-sg-inf-somej) edge node[pos=.5] {$\ref{proof:gap<->rad}$} (n-srlinf-p-0-somej);
\path[<->] (n-sg-inf-allj) edge node[pos=.5] {$\ref{proof:gap<->rad}$} (n-srlinf-p-0-allj);

\end{scope}

    \matrix [column sep=5mm,
    ] at (7.5,16){
  \node[shape=regular polygon,draw ,regular polygon sides=6,fill=red!20,text width=.3cm,scale=.7,label=right:\normalsize Geometric]{}; &
  \node[shape=circle,draw,fill=blue!15,text width=.3cm,label=right:\normalsize $V$-function]{}; &
  \node[shape=regular polygon,draw ,regular polygon sides=5,fill=brown!20,text width=.3cm,scale=.7,label=right:\normalsize Spectral] {}; &
  \node[shape=regular polygon,draw ,regular polygon sides=9,fill=purple!30,text width=.3cm,scale=.7,label=right:\normalsize $L^\infty_V$-norm] {}; &
  \node[shape=regular polygon,draw ,regular polygon sides=4,fill=green!15,text width=.3cm,scale=.7,label=right:\normalsize Reversible] {};
\\
};
\end{tikzpicture}
    \caption{Diagram illustrating which of this paper's results
(yellow edge labels) provide proofs of
implications between which of the different equivalent conditions (nodes).
(All arrows touching a green rectangle assume that the chain is reversible.)}
    \label{fig:statements}
\end{figure}

\section{Preliminary Lemmas}\label{sec-lemmas}

We begin with some preliminary lemmas, which are used freely in the
sequel, and can be referred to as needed.

\begin{lemma}\label{lem-small_sets_exist}
    Let $P$ be the transition kernel of a $\phi$-irreducible, aperiodic Markov chain with stationary distribution $\pi$ on a countably generated state space $\X$. Then for any
    measurable subset $A \subseteq \X$ such that $\pi(A) > 0$, there exists a small
    set $S$, such that $S\subseteq A$.
\end{lemma}

\begin{proof}
This result goes back to \cite{Doeblin1940,JainJamison, Orey}, and uses
that $\F$ is countably generated; see e.g.\
Theorems~5.2.1 and~5.2.2 in \cite{MeynTweedie}.
\end{proof}

\begin{lemma}\label{lem-total_variation_is_measurable}
    Let $P$ be the transition kernel of a $\phi$-irreducible, aperiodic Markov chain with stationary distribution $\pi$ on a countably generated state space $\X$. Then, the function $D_n : \X \to [0,\infty)$ defined by $D_n(x) = \tvnorm{P^n(x,\cdot) - \pi(\cdot)}$ is measurable.
\end{lemma}

\begin{proof}
This follows from \cite[Appendix]{hybrid}, which proves that for any
bounded signed measure $\nu(\cdot,A)$ on a countably generated space
such that the function
$x \mapsto \nu(x,A)$ is measurable for each fixed $A\in\mathcal{F}$,
the function $x \mapsto \sup\limits_{A\in\mathcal{F}}\nu(x,A)$
is also measurable.
\end{proof}

\ifx34
\begin{lemma}\label{lemma-Lp}
For each $1\leq p < \infty$, consider
\begin{equation*}
    L^{p}_{\pi}(\X) \coloneqq \{ f:\X \to \R : \int_{\X}|f|^{p}d\pi < \infty \}, \quad \quad \norm{f}_{p} \coloneqq \left(\int_{\X}|f|^{p}d\pi\right)^{1/p}
\end{equation*}

Then, for each $1\leq p < s < \infty$, $L^{s}_{\pi}(\X) \subseteq L^{p}_{\pi}(\X)$, and for all $f\in L^{s}_{\pi}(\X)$, $\norm{f}_{p} \leq \norm{f}_{s}$.
\end{lemma}

\begin{proof}
    If $1\leq p < s < \infty$ and $f\in L^{s}_{\pi}(\X)$, then $|f|^{p}\in L^{s/p}_{\pi}(\X)$ and 
    \begin{equation*}
        \norm{|f|^{p}}_{s/p} = \left( \int_{\X} |f|^{s}d\pi \right)^{p/s} = \norm{f}^{p}_{s}.
    \end{equation*}
    
    On the other hand, we have that $\one_{\X}\in L^{q}_{\pi}(\X)$ for all $q\in[1,\infty)$ and $\norm{\one_{\X}}_{q} = 1$. From H\"{o}lder's inequality, $|f|^{p} = \one_{\X}|f|^{p}$ is integrable and
    
    \begin{equation*}
        \norm{f}_{p}^{p} = \int_{\X} \one_{\X}|f|^{p}d\pi \leq \norm{\one_{\X}}_{s/(s-p)}\norm{|f|^{p}}_{s/p} = \norm{f}_{s}^{p}.  
    \end{equation*}
    
    From this, we can conclude that $f\in L^{p}_{\pi}(\X)$ and $\norm{f}_{p} \leq \norm{f}_{s}$.
\end{proof}
\fi

\begin{lemma}\label{lem-TVL1}
For probability measures $\mu_1$ and $\mu_2$,
    $\tvnorm{\mu_{1}-\mu_{2}} \, = \,
\half \, \norm{\mu_1-\mu_2}_{L^{1}(\pi)}$.
\end{lemma}

\begin{proof}
Recall that
    $\tvnorm{\mu_{1}-\mu_{2}} \, = \,
\sup\limits_{A\in\F}|\mu_{1}(A)-\mu_{2}(A)|$.
Let $\nu = \mu_1 + \mu_2$ so that $\mu_i \ll \nu$,
and let $f_i = {d\mu_i \over d\nu}$.  Then
$\mu_{1}(A)-\mu_{2}(A)
= \int_A [f_1(x)-f_2(x)] \, \nu(dx)$.
This is maximised when $A = A_+ := \{ x : f_1(x) > f_2(x) \}$, and its
negative takes the same maximum when $A = A_+^C$.  Hence,
$$
\tvnorm{\mu_{1}-\mu_{2}}
\ = \ \mu_1(A_+) - \mu_2(A_+)
\ = \ \int_{A_+} [f_1(x)-f_2(x)] \, \nu(dx).
$$
But then
\begin{align*}
\norm{\mu_1-\mu_2}_{L^{1}(\pi)}
&\ = \ (\mu_1-\mu_2)^+(\X) + (\mu_1-\mu_2)^-(\X) \\
&\ = \ \int_{A_+} [f_1(x) - f_2(x)] \, \nu(dx)
+ \int_{A_+^C} [f_2(x) - f_1(x)] \, \nu(dx) \\
&\ = \ 2 \, \int_{A_+} [f_1(x) - f_2(x)] \, \nu(dx) \\
&\ = \ 2 \, \tvnorm{\mu_{1}-\mu_{2}} .
\qedhere
\end{align*}
\end{proof}

\ifx34
\def\meddot{\bullet}
\def\disjunionrel{ {\buildrel \meddot \over \cup} }
\def\disjunion{ \, \disjunionrel \, }
Let $\rho = \mu_1-\mu_2$, and let
$\X = \X^+ \disjunion \X^-$ be a Hahn Decomposition for $\rho$,
so that $\rho(S) \ge 0$ for all $S \subseteq \X^+$
and $\rho(S) \le 0$ for all $S \subseteq \X^-$.
Then $\rho(A)$ is maximised when $A=\X^+$, and minimised when $A=\X^-$.
Hence,
$$
\tvnorm{\mu_{1}-\mu_{2}}
\ = \ \sup\limits_{A\in\F}|\rho(A)|
\ = \ \max\left[ \rho(\X^+), \, \rho(\X^-) \right]
$$
\fi

\begin{lemma}\label{lem-L1L2}
For any signed measure $\mu \ll \pi$, we have
$\norm{\mu}_{L^{1}(\pi)} \leq \norm{\mu}_{L^{2}(\pi)}$
(though one or both of those quantities might be infinite).
\end{lemma}

\begin{proof}
Recall the definition $\langle \mu,\nu \rangle
= \int_\X {d\mu \over d\pi} \, {d\nu \over d\pi} \, d\pi$.  Hence,
if $|\mu|$ is the measure with
${d|\mu| \over d\pi} = \big|{d\mu \over d\pi}\big|$, then
$\langle |\mu|,\pi \rangle
= \int_\X \big|{d\mu \over d\pi}\big| \, (1) \, d\pi
= \mu^+(\X) + \mu^-(\X) = \|\mu\|_{L^{1}(\pi)}$.
Also
$$
\norm{\, |\mu|\, }_{L^{2}(\pi)}
\ = \ \sqrt{ \int_\X \big|{d\mu \over d\pi}\big|^2 \, d\pi }
\ = \ \sqrt{ \int_\X \big({d\mu \over d\pi}\big)^2 \, d\pi }
\ = \ \norm{\mu}_{L^{2}(\pi)},
$$
and
$$
\norm{\pi}_{L^{2}(\pi)}
\ = \ \sqrt{ \int_\X \big({d\pi \over d\pi}\big)^2 \, d\pi }
\ = \ \sqrt{ \int_\X \big(1\big)^2 \, d\pi }
\ = \ 1.
$$
So, by the Cauchy-Schwarz inequality,
$$
\|\mu\|_{L^{1}(\pi)}
\ = \ \langle |\mu|,\pi \rangle
\ \le \ \norm{\, |\mu|\, }_{L^{2}(\pi)} \, \norm{\pi}_{L^{2}(\pi)}
\ = \ \norm{\mu}_{L^{2}(\pi)} \, (1)
\ = \ \norm{\mu}_{L^{2}(\pi)}.
\qedhere
$$
\end{proof}

\begin{lemma}\label{prop-Lp}
For all $1 \le p < s < \infty$,
we have $L^{s}(\pi) \subseteq L^{p}(\pi)$.
\end{lemma}

\begin{proof}
Let $1 \le p < s < \infty$, and let $\mu \in L^{s}(\pi)$ so
$\norm{\mu}_{L^{s}(\pi)} < \infty$.  Then,
$$
\norm{\mu}_{L^{p}(\pi)}^p
\ = \ \int_{\X} \left| \dfrac{d\mu}{d\pi} \right|^{p} d\pi
\ \le \ \int_{\X} \left( 1 + \left| \dfrac{d\mu}{d\pi} \right|^{s} \right) d\pi
\ = \ 1 + \norm{\mu}_{L^{s}(\pi)}^s
\ < \ \infty,
$$
so $\mu \in L^{p}(\pi)$.
\end{proof}

We next present some lemmas which mention spectra of operators.

\begin{lemma}\label{lem-directsum}
Suppose an operator $P$ on a Banach space $\vecsp$ can be decomposed
as a {\rm direct sum} $P = P_1 \oplus P_2$, where
$\vecsp = \vecsp_1 \times \vecsp_2$ and each $P_i$ is an operator
on $\vecsp_i$, meaning that $P(h_1,h_2) = (P_1 h_1,P_2 h_2)$
for all $h_1\in\vecsp_1$ and $h_2\in\vecsp_2$.  Then
$\s_{\vecsp}(P) = \s_{\vecsp_1}(P_1) \cup \s_{\vecsp_2}(P_2)$,
i.e.\ the spectrum of $P$ is the union of the spectra of the
sub-operators $P_1$ and $P_2$.
\end{lemma}

\begin{proof}
Since $P = P_1 \oplus P_2$, therefore $P$ has the block decomposition
$$
P \ = \
\begin{pmatrix}
P_1 & 0 \\
0 & P_2
\end{pmatrix}
$$
with respect to $\vecsp = \vecsp_1 \times \vecsp_2$.
If $\lambda \not\in \s_{\vecsp_1}(P_1) \cup \s_{\vecsp_2}(P_2)$,
then there are inverse operators $A_i$ on $\vecsp_i$ such that
$(\lambda I_i - P_i) A_i
= A_i (\lambda I_i - P_i) = I_i$ for $i=1,2$, whence
$(\lambda I - P) (A_1,A_2) =
(A_1,A_2) (\lambda I - P) = I_1 \oplus I_2 = I$,
so $\lambda \not\in \s_{\vecsp}(P)$.
Conversely, if $\lambda \not\in \s_{\vecsp}(P)$, then $\lambda I - P$
has some inverse operator, so in block form we have
$$
\begin{pmatrix}
\lambda I_1 - P_1 & 0 \\
0 & \lambda I_2 - P_2
\end{pmatrix}
\begin{pmatrix}
A & B \\
C & D
\end{pmatrix}
\ = \
\begin{pmatrix}
A & B \\
C & D
\end{pmatrix}
\begin{pmatrix}
\lambda I_1 - P_1 & 0 \\
0 & \lambda I_2 - P_2
\end{pmatrix}
\ = \ I
\ = \
\begin{pmatrix}
I_1 & 0 \\
0 & I_2
\end{pmatrix}
.
$$
It follows that $(\lambda I_1 - P_1) A = A (\lambda I_1 - P_1) = I_1$
and $(\lambda I_2 - P_2) D = D (\lambda I_2 - P_2) = I_2$, so that
$\lambda \not\in \s_{\vecsp_1}(P_1) \cup \s_{\vecsp_2}(P_2)$.
\end{proof}

\begin{lemma}\label{lem-P-Linf}
Let $P$ be
the transition kernel of a $\phi$-irreducible Markov chain with stationary distribution $\pi(\cdot)$, and let $V:\X \to [1,\infty]$ be a $\pi$-a.e.-finite measurable function. Then, the following hold:
\begin{enumerate}[label=\arabic*), ref=\textit{\roman*}]
    \item $|f|_{V} \leq 1$ if and only if $|f(x)| \leq V(x)$ for all
$x\in\X$.

\item
If there is $j\in\N$ with $\pi(V^j)<\infty$, then $\pi(V)<\infty$.
    \item If $P$ is a bounded operator on $\lv$, then
$\s_{\lv}(P) \setminus \{1\} \, \subseteq \, \s_{\lvz}(P)$.

\item The number~1 is an eigenvalue of $P$ with multiplicity 1,
regarding $P$ as an operator on
$L^{p}(\pi)$ for any $1 \le p < \infty$.
Furthermore, if $\pi(V^j)<\infty$ for some $j\in\N$, then this also holds
regarding $P$ as an operator on $\lv$ or $\lvz$.

\item If there are $\lambda<1$ and $b<\infty$ and a small set $S\in\F$ with 
$PV(x) \le \lambda \, V(x) + b \, \one_S(x)$ for all $x\in\X$,
then $\pi(V)<\infty$.

\end{enumerate}
\end{lemma}

\begin{proof}
\begin{enumerate}[label=\arabic*), ref=\textit{\roman*}]
    \item
    
    If $|f|_{V} \leq 1$, then for each $x\in \X$, 
    \begin{align*}
    \dfrac{|f(x)|}{V(x)}\leq |f|_V  \leq 1,
    \end{align*}
    from which we conclude that $|f|\leq V$.
Conversely, if $|f|\leq V$, then for each $x\in \X$,
$\dfrac{|f(x)|}{V(x)}\leq 1$, and thus
$|f|_{V} = \sup\limits_{x\in\X}\dfrac{|f(x)|}{V(x)}\leq 1$. 

\item This follows since we always have $V(x) \le V^j(x) + 1$.
[In fact, since $V \ge 1$, the ``$+1$'' is not actually necessary.]

    \item
Any $f\in \lv$ can be written as $f=f_0+c$ where $f_0\in
L_{V,0}^{\infty}$ and $c=\pi(f)$.  Then $Pf = Pf_0 + c$.
It follows that
$P$ has the direct sum representation $P = P_0 \oplus I_\R$,
where $I_\R$ is the identity operator on $\R$.
Hence, by Lemma~\ref{lem-directsum},
$\s_{\lv}(P) = \s_{\lvz}(P) \cup \s_{\R}(I_\R)
= \s_{\lvz}(P) \cup \{1\}$.
So, $\s_{\lv}(P) \setminus \{1\} \subseteq \s_{\lvz}(P)$, as claimed.

\item
Since $P$ is $\phi$-irreducible with stationary probability
measure~$\pi$, it follows that $P$ is ``positive'' as defined
on \cite[p.~235]{MeynTweedie}.
Hence, $P$ is recurrent by \cite[Proposition~10.1.1]{MeynTweedie}.
Then, \cite[Theorem~10.0.1]{MeynTweedie} shows that $\pi$ is unique,
i.e.\ $P$ has a unique invariant probability measure.
This implies by \cite[Proposition~22.1.2]{Douc} that~1
is an eigenvalue of $P$ with multiplicity 1 on any $L^p(\pi)$ space.
Furthermore, if $\pi(V^j)<\infty$, then
$|f| \le C V$ implies that $\pi(|f|^j) \le C^j \pi(V^j) < \infty$,
so in that case $\lv$ and $\lvz$ are subspaces of $L^j(\pi)$,
and hence the result holds on $\lv$ and $\lvz$ too.

\item
The implication ``$(iii) \implies (i)$'' of~\cite[Theorem~14.0.1]{MeynTweedie}
with the choice $f(x) = (1-\lambda) \, V(x)$ shows that
$\pi(f)<\infty$,
i.e.\ $(1-\lambda) \, \pi(V)<\infty$,
hence $\pi(V)<\infty$.
[In fact, once we know that $\pi(V)<\infty$, then since
$PV \le \lambda \, V + b$, it follows that
$\pi(PV) \le \pi(\lambda V + b)$,
i.e.\ $\pi(V) \le \lambda \, \pi(V) + b$,
and hence $\pi(V) \le b \, \big/ (1-\lambda)$.]
\qedhere
\end{enumerate}
\end{proof}

\ifextralemmas

\begin{lemma}\label{lem-eigenmult}
Let $P$ be the transition kernel of a $\phi$-irreducible Markov chain
on a state space $\X$,
and let $\vecsp$ be either
a vector space of signed measures on $\X$ which includes $\pi$,
or a vector space of functions $\X\to\R$ which includes the constant
functions.  Then $P$ as an operator on $\vecsp$ has eigenvalue~1 with
multiplicity~1.
\end{lemma}

\begin{proof}
It follows that 1 is an eigenvalue since $\pi P = \pi$
and also $Pc=c$ for constant functions $c$.
To show multiplicity~1, suppose first that there is a
signed measure $\pi'$, linearly independent from $\pi$, such
that also $\pi' P = \pi'$.  
Then there is some linear combination $\mu = a\pi + b\pi'$ such
that both $\mu^+$ and $\mu^-$ are non-zero measures.

\ifx34
Set $\nu = \pi + \pi'$, so that $\pi \ll \nu$ and $\pi' \ll \nu$
if f^+ = {d\mu^+ \over d\pi}
and f^- = {d\mu^- \over d\pi}, then P(f^+ - f^-) = f^+ - f^-,
so \pi( P^n(f^+ - f^-) ) = \pi( f^+ - f^- ).  But for some n, create
overlap, get cancellation, get contradiction
\fi
\end{proof}

\medskip
Finally, we present some lemmas which are specific to reversible chains.

\begin{lemma}\label{lem-reversibility-alln}
If $P$ is a transition kernel of a reversible Markov chain,
then for all $n\geq 1$,
\begin{equation}\label{eq:reversibility-alln}
\pi(dx) \, P^{n}(x,dy) \ = \ \pi(dy) \, P^{n}(y,dx).
\end{equation}
\end{lemma}

\begin{proof}
For $n=1$, \pref{eq:reversibility-alln} holds by reversibility.
Furthermore, if \pref{eq:reversibility-alln} holds for $n$,
then for $n+1$,
\begin{align*}
    \pi(dx) P^{n+1} (x,dy) = \pi (dx) \int_{z\in\X} P(z,dy) P^{n}(x,dz) 
    &= \int_{z\in\X} P(z,dy) \pi (dx) P^{n}(x,dz)\\ 
\mbox{(Induction hypothesis)} \ \
		&= \int_{z\in\X} P(z,dy) \pi (dz) P^{n}(z,dx)\\
\mbox{(Reversibility)} \ \ &= \int_{z\in\X} P^{n}(z,dx) \pi (dy) P(y,dz)\\
    &= \, \pi (dy) \, P^{n+1}(y,dx),
\end{align*}
so \pref{eq:reversibility-alln} holds for $n+1$.
Hence, by induction, \pref{eq:reversibility-alln} holds for all $n \ge 1$.
\end{proof}

\begin{lemma}\label{lem-derivate-l2}
If P is a transition kernel of a reversible Markov chain with stationary distribution $\pi$, then for each $\mu\in\lpp$ and each $n\in \N$,
$$
    \dfrac{d(\mu P^{n})}{d\pi} = P^{n}\dfrac{d\mu}{d\pi}.
$$
\end{lemma}

\begin{proof}
For all $A \in \F$ and $n\in\N$,
\begin{align*}
\mu P^n(A)
\ = \ \int_{x\in\X} P^n(x,A) \, \mu(dx)
\ &= \ \int_{x\in\X}\int_{y\in A} \dfrac{d\mu}{d\pi}(x) \, P^n(x,dy) \, \pi(dx)\\
\mbox{(Fubini's Theorem)} \
\ &= \ \int_{y\in A}\int_{x\in\X} \dfrac{d\mu}{d\pi}(x) \, \pi(dx) \, P^n(x,dy)\\
\mbox{(Lemma \ref{lem-reversibility-alln})} \
\ &= \ \int_{y\in A}\int_{x\in\X} \dfrac{d\mu}{d\pi}(x) \, \pi(dy) \, P^n(y,dx)\\
&= \ \int_{y\in A}
\Big( \int_{x\in\X} \dfrac{d\mu}{d\pi}(x) \, P^n(y,dx) \Big)
\, \pi(dy)\\
&= \ \int_{y\in A} \Big( P^n\dfrac{d\mu}{d\pi}(y) \Big) \, \pi(dy).
\end{align*}
Since this is true for all $A\in\F$, the result follows.
\end{proof}

\fi

\begin{lemma}\label{lem-P-L2}
Let $P$ be the transition kernel of a reversible Markov chain with stationary distribution $\pi$, such that $P$ is a bounded operator on $\lpp$. Then, the following holds:
\begin{enumerate}[label=\arabic*), ref=\textit{\roman*}]
	\item The operator $P - \Pi$ is self-adjoint.
	\item For each $\mu \in L^{2}(\pi)$,
the signed measure $\mu - \mu(\X)\pi$ is orthogonal to $\pi$.
	\item For each $\mu \in \lpp$, $\norm{\mu - \mu(\X)\pi}_{\lpp}^2 = \norm{\mu}_{\lpp}^2 - \mu(\X)^2$.
	\item $\s_{\lpp}(P)\setminus \{1\}
\, \subseteq \, \s_{\pi^\bot}(P)$.
\end{enumerate}
\end{lemma}

\begin{proof}
\begin{enumerate}[label=\arabic*), ref=\textit{\roman*}]

	\item For $\mu,\nu \in \lpp$, we have
	    $\langle \mu(P - \Pi), \nu \rangle
	\, = \, \langle \mu P, \nu \rangle - \langle \mu \Pi, \nu \rangle$.
Now, since $P$ is reversible, it is self-adjoint on $\lpp$, so
$\langle \mu P, \nu \rangle = \langle \nu P, \mu \rangle$.
Also, we compute that
$\langle \mu \Pi, \nu \rangle
= \langle \mu(\X) \pi, \nu \rangle
= \mu(\X) \, \nu(\X)
= \langle \nu \Pi, \mu \rangle$.
Hence,
$\langle \mu(P - \Pi), \nu \rangle
= \langle \nu(P - \Pi), \mu \rangle$,
so $P - \Pi$ is self-adjoint.
	
	\item Let $\mu \in L^{2}(\pi)$, then,
	\[
	\langle \mu - \mu(\X)\pi, \pi \rangle
= \langle \mu,\pi\rangle - \mu(\X)\langle \pi, \pi \rangle
= \langle \mu,\pi\rangle - \langle \mu,\pi\rangle \norm{\pi}_{\lpp}
= \langle \mu,\pi\rangle - \langle \mu,\pi\rangle = 0.
	\]

	\item Let $\mu\in \lpp$. Then,
	\begin{align*}
	    \norm{\mu - \mu(\X)\pi}_{\lpp}^2 
	        &= \int_{\X}\left| \dfrac{d\mu}{d\pi}(y) - 
\mu(\X) \, (1) \right|^2\pi(dy)\\
	        &= \int_{\X} \left[ \left(\dfrac{d\mu}{d\pi}(y)\right)^2
- 2 \, \mu(\X)\dfrac{d\mu}{d\pi}(y) + \mu(\X)^2\right]\pi(dy)\\
	        &= \norm{\mu}_{\lpp}^2 - 2 \mu(\X)^2 + \mu(\X)^2 \\
	        &= \norm{\mu}_{\lpp}^2 - \mu(\X)^2.
	\end{align*}

    \item
Any signed measure $\mu\in\lpp$ can be decomposed as
$\mu = \mu_0 + c \, \pi$, where $c= \langle \mu,\pi \rangle = \mu(\X)$,
and $\langle \mu_0,\pi \rangle = \mu_0(\X) = 0$, so $\mu_0 \in \pi^\bot$.
Then $\mu P = \mu_0 P + c \, \pi$.  It follows that
$P$ has the direct sum representation $P = P|_{\pi^{\bot}} \oplus I_\R$
with respect to $\lpp = \pi^\bot \times \R$.
Hence, by Lemma~\ref{lem-directsum},
$\s_{\lpp}(P) = \s_{\pi^\bot}(P) \cup \{1\}$,
so $\s_{\lpp}(P) \setminus \{1\} \subseteq \s_{\pi^\bot}(P)$, as claimed.
\end{enumerate}
\end{proof}

\begin{lemma}\label{lem-equality-l2}
Let $P$ be a transition kernel from a reversible Markov chain with stationary distribution $\pi$. Then,
\[
\norm{P -\Pi}_{\lpp}
\ = \
\norm{P}_{\pi^\bot}.
\]
\end{lemma}

\begin{proof}
Any $\mu\in\lpp$ can be written as $\mu = \mu_0 + c \, \pi$,
where $c=\mu(\X)$ and $\mu_0 \in \pi^\bot$ so $\mu_0(\X)=0$.
Then $\mu_0 \Pi = \mu_0(\X) \, \pi = 0$, so
$$
\mu(P-\Pi)
\ = \ (\mu_0+c \, \pi)(P-\Pi)
\ = \ \mu_0 P + c \, \pi - 0 - c \, \pi
\ = \ \mu_0 P .
$$
Also $\norm{\mu}_{\lpp} = 
\norm{\mu_0}_{\lpp} + c^2
\ge \norm{\mu_0}_{\lpp}$.
Hence,
$$
\norm{P-\Pi}_{\lpp}
\ = \ \sup_{0 < \norm{\mu}_{\lpp} < \infty}
	{||\mu(P-\Pi)||_{\lpp} \over \norm{\mu}_{\lpp}}
\ = \ \sup_{0 < \norm{\mu}_{\lpp} < \infty}
	{||\mu_0 P||_{\lpp} \over \norm{\mu_0}_{\lpp} + c^2}
.
$$
This supremum is achieved when $c=0$, i.e.\ when $\mu=\mu_0 \in \pi^\bot$,
so that
\[
\norm{P-\Pi}_{\lpp}
\ = \ \sup_{0 < \norm{\mu_0}_{\lpp} < \infty \atop \mu_0 \in \pi^\bot}
	{||\mu_0 P||_{\lpp} \over \norm{\mu_0}_{\lpp}}
\ = \ \norm{P}_{\pi^\bot}.
\qedhere
\]
\end{proof}

\section{Proofs for Geometric Conditions}
\label{sec-geometric}

We now begin proving the actual equivalences of the various conditions in
Theorem~\ref{thm-equiv}, as per the plan illustrated in
Figure~\ref{fig:statements}.  We begin with some results related to
some of the ``geometric'' conditions.

\begin{proposition}\label{proof:iii->iv}
\pref{eq-gefalp} $\implies$ \pref{eq-gefslp}.
\end{proposition}

\begin{proof}
Immediate upon e.g.\ choosing
$p=2$ and setting $C_{\mu} = C_{2,\mu}$ and $\rho_\mu = \rho_2$
for each probability measure $\mu \in \lpp$.
\end{proof}

\begin{proposition}\label{proof:iv->v}
\pref{eq-gefslp} $\implies$ \pref{eq-ssge}.
\end{proposition}

\begin{proof}
By Lemma \ref{lem-small_sets_exist}, there exists a
small set $S \subset \X$. Since
    by assumption $P$ is geometrically ergodic starting from all probability
    measures in $L^p(\pi)$ it suffices to show that $\pi_S \in L^p(\pi)$. Now
    for any measurable $A \subset \X$ we have,
    \begin{align*}
        \pi_S(A)
        = \dfrac{\pi(S \cap A)}{\pi(S)}
        = \frac{1}{\pi(S)}\int_{A}\one_{S}d\pi
    \end{align*}
    which implies $\frac{d\pi_S}{d\pi} = \one_{S}/\pi(S)$. Thus
    \begin{align*}
        \int_{\X}\left|\frac{d\pi_S}{d\pi}\right|^pd\pi
        = \int_{\X}\frac{1}{\pi(S)^p}\one_{S}d\pi
        = \frac{1}{\pi(S)^{p-1}}
        < \infty
    \end{align*}
\end{proof}

\begin{proposition}\label{proof:v->i}
\pref{eq-ssge} $\implies$ \pref{eq-geometric-erg}.
\end{proposition}

\begin{proof}
This is the result of \cite[Theorem~1]{nummelinTweedie:1978},
which generalizes the countable state space result of \cite{vere-jones1962}.
\end{proof}

\begin{proposition}\label{proof:i->ii}
\pref{eq-geometric-erg} $\implies$ \pref{eq-weak-geometric-erg}.
\end{proposition}

\begin{proof}
Immediate upon choosing $A=\X$, and $\rho_x=\rho$ for all $x\in\X$.
\end{proof}

\begin{proposition}\label{proof:ii->vsup}
    \pref{eq-weak-geometric-erg} $\implies$ \pref{eq-ssgesup}.
\end{proposition}

\begin{proof}
    Let $A\in\F$ with $\pi(A) > 0$ and
    $\tvnorm{P^n(x,\cdot) - \pi(\cdot)} \ \leq \ C_x \, \rho_x^n$
for all $x\in A$ and $n\in\N$.
For each $n \in \N$,
    let $D_n : A \to [0,\infty)$ by $D_n(x) = \tvnorm{P^n(x,\cdot) -
    \pi(\cdot)}$. Then each $D_n$ is measurable by Lemma
    \ref{lem-total_variation_is_measurable}, hence so are
the functions $r,s,M : A \to [0,\infty]$ defined by
    \begin{equation*}
        r(x) \, = \, \limsup_{n \to \infty} \, [D_n(x)^{1/n}], \quad
        s(x) \, = \, [r(x) + 1]/2, \quad
        M(x) \, = \, \sup_{n} \, [D_n(x)/s(x)^n].
    \end{equation*}
In particular, for each $n\in\N$, we have
$M(x) \ge D_n(x)/s(x)^n$, hence $D_n(x) \le M(x) \, s(x)^n$.

Next, note that \pref{eq-weak-geometric-erg} says that
for each $x\in A$,
$D_n(x) \le C_x \, \rho_x^n$,
so $r(x) \le \limsup_{n\to\infty} [C_x \, \rho_x^n]^{1/n} = \rho_x < 1$.
Hence $r(x) < s(x) < 1$.
In particular,
$\limsup_{n \to \infty} \, [D_n(x)^{1/n}] < s(x)$.
Hence, there is $N(x)\in\N$ such that
for all $n > N(x)$ we have
$D_n(x)^{1/n} < s(x)$, i.e.\ $D_n(x)/s(x)^n < 1$.
Then
$$
M(x) \ \le \
\max\left[ D_1(x)/s(x)^1, \  D_2(x)/s(x)^2, \
     \ldots, \  D_{N(x)}(x)/s(x)^{N(x)}, \  1\right]
\ < \ \infty
\, .
$$

Now, since $s$ and $M$ are measurable, so are the nested subsets
    \begin{align*}
        B_k \ &:= \ \{x \in A :
s(x) \leq 1 - \frac{1}{k}, \ M(x) \leq k\},
        \qquad k \in \N.
    \end{align*}
Since $s(x)<1$ and $M(x)<\infty$ for each $x\in A$, we must have
$\bigcup_k B_k = A$.  Continuity of measures then
    implies that $\lim_{k \to \infty} \pi(B_k) = \pi(A) > 0$, so
    there is $K \in \N$ with $\pi(B_K) > 0$.
By Lemma~\ref{lem-small_sets_exist}, there exists a small set
    $S \subseteq B_K$.
Then for $x\in S$, we have $x \in B_K$, so
$s(x) \le 1 - \frac1K$ and $M(x) \le K$.
It follows that for $x \in S$ and $n \in \N$,
    \begin{align*}
    \tvnorm{P^n(x,\cdot) - \pi(\cdot)} \ = \ D_n(x)
\ \leq \ M(x) \, s(x)^n \ \leq \ K \, \Big(1 - \frac{1}{K}\Big)^n .
    \end{align*}
This establishes \pref{eq-ssgesup}
with $C_S=K$ and $\rho_S=1 - \frac{1}{K}$.
\end{proof}

\begin{proposition}\label{proof:ii->v}
    \pref{eq-ssgesup} $\implies$ \pref{eq-ssge}.
\end{proposition}

\begin{proof}
This follows since
    \begin{align*}
        \tvnorm{\pi_S P^n(\cdot) - \pi(\cdot)}
        \ &= \ \sup_{D}|\pi_S P^n(D) - \pi(D)| \\
	\ &= \ \sup_{D} |\frac{1}{\pi(S)}\int_{S} [P^n(x,D)-\pi(D)] \pi(dx)| \\
        \ &\leq \ \sup_{D}\sup_{x \in S}|P^n(x,D) - \pi(D)| \\
        \ &= \ \sup_{x \in S}\tvnorm{P^n(x,\cdot) - \pi(\cdot)} \\
        \ &\leq \ C_S \, \rho_S^n.
\qedhere
    \end{align*}
\end{proof}

\begin{proposition}\label{proof:v->xxiv}
\pref{eq-ssge} $\implies$ \pref{eq-tau-c}.
\end{proposition}

\begin{proof}
This is the content of the ``$(i) \implies (ii)$''
implication of \cite[Theorem~15.0.1]{MeynTweedie}.
\end{proof}

\section{Proofs for $V$-function and $L^\infty_V$ Conditions}
\label{sec-V}

\begin{proposition}\label{proof:choosejone}
    $\pref{eq-vinf-allj} \implies \pref{eq-vinf-somej}$.
\end{proposition}

\begin{proof}
Immediate (just choose $j=1$).
\end{proof}

\begin{proposition}\label{proof:viii->xiv}
    $\pref{eq-vuex-allj} \implies \pref{eq-vinf-allj}$.
\end{proposition}

\begin{proof}
    Let $f \in L^\infty_V$ such that $|f|_V =
    1$. Then, $|f| \leq V$, and, if $\pref{eq-vuex-allj}$ holds, then
for each $x \in \X$ and each $n\in\N$,
    \begin{align*}
        |P^nf(x) - \Pi(f)(x)| = |P^nf(x) - \pi(f)|
        &\leq \sup\limits_{|f|\leq V} |P^n(x,f) - \pi(f)|
        \leq CV(x)\rho^n,
    \end{align*}
    which implies
    \begin{align*}
        |(P^n - \Pi)f|_V
        &= \sup_{x \in \X}\frac{|P^n f(x) - \Pi(f)(x)|}{V(x)}
        = \sup_{x \in \X}\frac{|P^n f(x) - \pi(f)|}{V(x)}
        \leq C\rho^n,
    \end{align*}
    and therefore,
    \begin{align*}
        \norm{P^n - \Pi}_{L^\infty_V}
        &= \sup_{\substack{f \in L^\infty_V \\ |f|_V = 1}}|(P^n - \Pi)f|_V
        \leq C\rho^n.
\qedhere
    \end{align*}
\end{proof}

\begin{proposition}\label{proof:xiv<->xv}
    \pref{eq-vinf-somej} $\Leftrightarrow$ \pref{eq-vinf0-somej}, \and
    \pref{eq-vinf-allj} $\Leftrightarrow$ \pref{eq-vinf0-allj}.
\end{proposition}

\begin{proof}
($\Rightarrow$) Let $n \in \N$. Given that $\lvz \subseteq \lv$,     
$\norm{P^{n}}_{\lvz} \leq \norm{P^{n} - \Pi}_{\lv} \leq C \, \rho^n$.

($\Leftarrow$) If $f\in \lv$ such that $|f|_{V} = 1$, we have
\begin{align*}
|(P^{n} - \Pi)f|_{V} &= |(P^{n} - \one_{\X}\otimes P^{n} \pi)f|_{V} \\
        &= |P^{n}f - (P^{n} \pi) f|_{V} \\
        &= |P^{n}(f -\pi (f))|_{V}\\
        &\leq |P^{n}(f -\pi (f))|_{V}\\
\mbox{($f -\pi(f) \in \lvz$)} \ \
		&\leq \norm{P^{n}}_{\lvz} \, |f -\pi (f))|_{V}\\
        &\leq \norm{P^{n}}_{\lvz} \, (|f|_{V} +|\pi (f)|_{V})\\
        &\leq C \, \rho^{n} \, (1 +\pi (V))\\
        &\leq C' \, \rho^{n},
\end{align*}
where $C' = C(1 + \pi(V)) < \infty$.
\end{proof}

\begin{proposition}\label{proof:xiii->xii}
\pref{eq-voinf-somej} $\Leftrightarrow$ \pref{eq-nlinf-p-pi-somej}, \and
\pref{eq-voinf-allj} $\Leftrightarrow$ \pref{eq-nlinf-p-pi-allj}.
\end{proposition}

\begin{proof}
    If $\norm{P^m - \Pi}_{L^\infty_V} < 1$, then for
    $f \in L^\infty_{V,0}$,
    \begin{align*}
        \norm{P^m}_{L^\infty_{V,0}} = \sup_{\substack{f \in L^\infty_{V,0} \\
        |f|_V = 1}}|P^mf|_V
        = \sup_{\substack{f \in L^\infty_{V,0} \\|f|_V = 1}}|(P^m - \Pi)f|_V \leq \sup_{\substack{f \in L^\infty_V \\|f|_V = 1}}|(P^m - \Pi)f|_V
        = \norm{P^m - \Pi}_{L^\infty_V} < 1,
    \end{align*}
    so that also $\norm{P^m}_{L^\infty_{V,0}} < 1$.

Conversely, if $s = \norm{P^{m}}_{\lvz} < 1$,
then for $f$ with $|f|_{V}\leq 1$ and $k\in \N$,
\begin{align*}
    |(P^{mk}-\Pi)f|_{V}
&= |P^{mk}(f-\pi(f))|_{V} \\
    &\leq  \norm{P^{mk}}_{\lvz}|f-\pi (f)|_{V}\\
    &=  \norm{(P^{m})^{k}}_{\lvz}|f-\pi (f)|_{V}\\
    &\leq s^{k}|f-\pi(f)|_{V}\\
    &\leq s^{k}(|f|_{V} + |\pi(f)|_{V})\\
    &\leq s^{k}(1 + \pi(V)).
\end{align*}
From Lemma~\ref{lem-P-Linf}, we must have $\pi (V) < \infty$.
Hence, there exists $n\in \N$ such that $s^{n}(1+\pi (V)) < 1$.
Thus, taking $m^{*} = mn$, we have that, for each $|f|_{V}\leq 1$,
\[
|(P^{m^{*}}-\Pi)f|_{V} < 1
\]
and therefore, $\norm{P^{m^{*}} - \Pi}_{\lv} < 1$.
\end{proof}

\begin{proposition}\label{proof:vi'<->vi}
\pref{eq-dc-nomom} $\implies$ \pref{eq-dc-allj}.
\end{proposition}

\begin{proof}
First of all, we must have $\pi(V)<\infty$
by Lemma~\ref{lem-P-Linf}.
Then, given $j\in \N$, let $\widehat{V} = V^{1/j}$, so
$\pi(\hat{V}^{j}) = \pi(V) < \infty$.
It follows from Jensen's inequality and concavity that
\begin{align*}
    P \, \widehat{V} \ \leq \ (PV)^{1/j} \ \leq \ (\lambda V + b\one_{S})^{1/j}
\ \leq \ \hat{\lambda} \, \hat{V} + \hat{b} \, \one_{S},
\end{align*}
with $\hat{\lambda}=\lambda^{1/j} < 1$ and $\hat{b}=b^{1/j} < \infty$,
thus showing \pref{eq-dc-nomom}.
\end{proof}

\begin{proposition}\label{proof:6.7a}
\pref{eq-nlinf-p-pi-somej} $\Leftrightarrow$ \pref{eq-vinf-somej},
\and
\pref{eq-nlinf-p-pi-allj} $\Leftrightarrow$ \pref{eq-vinf-allj}.
\end{proposition}

\begin{proof}
Suppose first that $s=\norm{P^{m} - \Pi}_{L^{\infty}_{V}} < 1$
for some $m\in\N$.
Let $\alpha = \norm{P - \Pi}_{L^{\infty}_{V}}$,
and let $n\in \N$. If $n\leq m$, we have that
    \begin{align*}
        \norm{P^{n} - \Pi}_{L^{\infty}_{V}} = \norm{(P - \Pi)^{n}}_{L^{\infty}_{V}} 
        &\leq \alpha^{n} \leq \alpha^{n}s^{-1}(s^{1/m})^{n}.
    \end{align*}
If $n > m$, then $n = mt + \ell$, for some $t\in\N$ and $0\leq \ell<m$, and hence
    \begin{align*}
        \norm{P^{n} - \Pi}_{\lv} 
            &= \norm{(P - \Pi)^{mt+\ell}}_{\lv}
            \leq \alpha^{l}\norm{(P - \Pi)^{mt}}_{\lv} \\
            &= \alpha^{l} \norm{(P^{m} - \Pi)^{t}}_{\lv}
            \leq \alpha^{l} s^{t}
            \leq \alpha^{l} s^{-1} (s^{1/m})^{n}.
    \end{align*}
So, taking $C = \max\limits_{1\leq r\leq m} \alpha^{r}s^{-1}$
and $\rho = s^{1/m}$, we conclude that for each $n\in\N$,
\[
\norm{P^{n} - \Pi}_{\lv} \leq C\rho^{n}.
\]

Conversely, if 
$\norm{P^{n} - \Pi}_{\lv} \leq C\rho^{n}$ for all $n\in\N$, then
we can simply choose a large enough~$m\in\N$ that $C \rho^m < 1$, to
obtain that
$\norm{P^{m} - \Pi}_{L^{\infty}_{V}} \le C \rho^m < 1$.
\end{proof}

\begin{proposition}\label{proof:6.7b}
\pref{eq-vinf-somej} $\implies$ \pref{eq-vuex-nomom}.
\end{proposition}

\begin{proof}
Since $\norm{P^{n} - \Pi}_{\lv} \leq C\rho^{n}$,
we have for each $n\in \N$ and $|f|_{V}\leq 1$ that
\[
|(P^{n} - \Pi ) f|_{V} \leq \norm{P^{n} - \Pi}_{\lv}|f|_{V} \leq C\rho^{n}|f|_{V} \leq C\rho^{n}.
\]
Hence, for each $n\in \N$, $|f|_{V} \leq 1$ and $x\in \X$,
\[
\dfrac{|P^{n}f(x) -  \pi (f)|}{V(x)} = \dfrac{|P^{n}f(x) -  \Pi (f)(x)|}{V(x)} \leq C\rho^{n}.
\]
By Lemma~\ref{lem-P-Linf},
$|f|_{V}\leq 1$ $\Leftrightarrow$ $|f|\leq V$, so
for each $n\in\N$ and $x\in \X$,
\begin{align*}
\sup\limits_{|f|\leq V} |P^{n}f(x) - \pi (f)|
&= \sup\limits_{|f|_V \leq 1} |P^{n}f(x) - \pi (f)|
\leq CV(x)\rho^{n}.
\qedhere
\end{align*}
\end{proof}

\begin{proposition}\label{proof:viii->vii}
\pref{eq-vuex-nomom} $\implies$ \pref{eq-vuemu-nomom},
\and
\pref{eq-vuex-allj} $\implies$ \pref{eq-vuemu-allj}.
\end{proposition}

\begin{proof}
This follows from the triangle inequality.
If $\mu(V) < \infty$ and $|f|\leq V$, then
\begin{align*}
    |\mu P^{n}f - \pi (f)| = \left| \int_{\X}P^{n}f(y) \mu(dy) - \pi (f) \right| &= \left| \int_{\X}P^{n}f(y) \mu(dy) - \int_{\X} \pi (f) \mu(dy) \right|\\
    &\leq  \int_{\X} \left| P^{n}f(y) - \pi (f) \right| \mu(dy)\\
    &\leq  \int_{\X} \sup\limits_{|f|\leq V} \left| P^{n}f(y) - \pi (f) \right| \mu(dy)\\
    &\le  \int_{\X} C \, V(y) \, \rho^{n} \, \mu(dy)\\
    &=  C \, \mu (V) \, \rho^{n}.
\end{align*}
Hence, $\sup\limits_{|f|\leq V} |\mu P^{n}f - \pi (f)|
\leq C \, \mu(V) \, \rho^{n}$ for all $n\in\N$.
\end{proof}

\begin{proposition}\label{proof:MT15}
\pref{eq-dc-allj} $\implies$ \pref{eq-vuex-allj}.
\end{proposition}

\begin{proof}
This is the content of \cite[Theorem~1]{nummelinTweedie:1978},
following \cite{vere-jones1962}; proofs also appear in
\cite[Theorem~15.0.1(iii)]{MeynTweedie} and \cite[Theorem~9]{probsurv}.
And since the same function $V$ is used in both conditions,
its moments are preserved.
\end{proof}

\begin{proposition}\label{proof:mutoge}
\pref{eq-vuemu-nomom} $\implies$ \pref{eq-geometric-erg}.
\end{proposition}

\begin{proof}
Let $\mu$ be a point-mass at $x$, so that $\mu(A)=1$ if $x\in A$
otherwise $\mu(A)=0$.  Then $\mu(V) = V(x)$, so
from \pref{eq-geometric-erg},
$$
\tvnorm{P^n(x,\cdot) - \pi(\cdot)}
\ = \ \tvnorm{\mu P^n(\cdot) - \pi(\cdot)}
\ = \ \sup_{|f| \le 1} \big| \mu P^n(f) - \pi(f) \big|
$$
$$
\ \le \ \sup_{|f| \le V} \big| \mu P^n(f) - \pi(f) \big|
\ \le \ C \, \mu(V) \, \rho^n
\ = \ C \, V(x) \, \rho^n
\, .
$$
Hence, \pref{eq-geometric-erg} holds with $C_x = C \, V(x)$.
\end{proof}

\begin{proposition}\label{proof:xxiv->vi}
\pref{eq-tau-c} $\implies$ \pref{eq-dc-nomom}.
\end{proposition}

\begin{proof}
The existence of a drift function $V$ satisfying the condition
\pref{eq-dc-nomom} follows from \cite[Theorem 15.2.4]{MeynTweedie}.
\end{proof}

\begin{proposition}\label{proof:vii->iii}
\pref{eq-vuemu-allj} $\implies$ \pref{eq-gefalp}.
\end{proposition}

\begin{proof}
Let $p\in (1,\infty)$, and let $\mu \in \lppi$ be a probability measure.
Let $j\in \N$ be large enough that $1 + \frac{1}{j} < p$,
so that $\mu \in L^{1+\frac{1}{j}}(\pi)$ by Lemma~\ref{prop-Lp}.
Then choose $V$ in \pref{eq-vuemu-allj} such that $\pi(V^{j+1})<\infty$.
Then, using the notation
$$
\|f\|_r \ := \ \left( \int_\X |f|^r \, d\pi \right)^{1/r}
$$
for functions $f:\X\to\R$,
since ${1 \over j+1} + {1 \over 1 + {1 \over j}} = 1$,
we have by H\"{o}lder's inequality that
\begin{align*}
\mu(V)
&\ = \ \int_{\X} V(x) \, \mu(dx)
\ = \ \int_{\X} V(x) \left( \dfrac{d\mu}{d\pi}(x) \right) \pi(dx) \\
&\ \leq \ \norm{V}_{j+1} \, \Big\| \dfrac{d\mu}{d\pi} \Big\|_{1+\frac{1}{j}}
= \ \pi(V^{j+1})^{1/(j+1)} \ \norm{\mu}_{L^{1+\frac{1}{j}}(\pi)}
\ < \ \infty.
\end{align*}
Then,
\begin{align*}
\tvnorm{\mu P^n(\cdot) - \pi(\cdot)}
&\ = \ \dfrac{1}{2} \, \sup_{|f| \leq 1}|\mu P^n(f) - \pi(f)| \\
&\ \leq \ \dfrac{1}{2} \, \sup_{|f| \leq V}|\mu P^n(f) - \pi(f)| \\
&\ \leq \ \dfrac{1}{2} \, C \, \mu(V) \, \rho^n,
\end{align*}
so \pref{eq-gefalp} holds
with $C_{p,\mu} = \dfrac{1}{2} \, C \, \mu(V) < \infty$.
\end{proof}

\section{Proofs for Spectral Conditions}
\label{sec-spectral}

\begin{proposition}\label{proof:gap<->rad}
\pref{eq-sg-inf-somej} $\Leftrightarrow$ \pref{eq-srlinf-p-0-somej},
\and
\pref{eq-sg-inf-allj} $\Leftrightarrow$ \pref{eq-srlinf-p-0-allj}.
\end{proposition}

\begin{proof}
($\Rightarrow$)
Since 1 is an eigenvalue with multiplicity~1
by Lemma~\ref{lem-P-Linf},
with corresponding eigenvectors
the non-zero constant functions which are not in $\lvz$, we must have
$\s_{\lvz}(P)
\subseteq \s_{\lv}(P) \setminus \{1\}$.  So,
if \pref{eq-sg-inf-somej} holds, then
$\s_{\lvz}(P)
\subseteq \s_{\lv}(P) \setminus \{1\}
\subseteq \{ z \in \mathbb{C} : |z| \leq \rho \}$ for some $\rho < 1$.
This implies that $r_{\lvz}(P) \le \rho < 1$.

($\Leftarrow$)
If $\rho := r_{\lvz}(P)<1$, then since
$\s_{\lv}(P)\setminus \{1\} \subseteq \s_{\lvz}(P)$
by Lemma~\ref{lem-P-Linf}, we have
\[
\s_{\lv}(P)\setminus \{1\} \ \subseteq \ \s_{\lvz}(P)
\ \subseteq \ \{z\in \mathbb{C} : |z| \leq \rho\}.
\qedhere
\]
\end{proof}

\begin{proposition}\label{proof:xi<->xiii}
\pref{eq-srlinf-p-0-somej} $\Leftrightarrow$ \pref{eq-voinf-somej}, \and
\pref{eq-srlinf-p-0-allj} $\Leftrightarrow$ \pref{eq-voinf-allj}.
\end{proposition}

\begin{proof}
($\Rightarrow$) By the spectral radius formula
(\cite{rudin:1991functional}, Theorem 10.13), $\rho =
r(P|_{L^{\infty}_{V,0}}) = \inf\limits_{n\geq 1}
\norm{P^{n}}_{\lvz}^{1/n}$. Hence, for any $\rho_{0} < 1$ with $\rho <
\rho_0$, there exists $m \in \N$ such that
$\norm{P^{m}}_{\lvz} < \rho_0^m < 1$.

\noindent ($\Leftarrow$) If $\norm{P^{m}}_{\lvz} < 1$ for some $m\in\N$,
\begin{align*}
r(P|_{L^{\infty}_{V,0}}) &= \inf\limits_{n\geq 1} \norm{P^{n}}_{\lvz}^{1/n} \leq \norm{P^{m}}_{\lvz}^{1/m} < 1,
\end{align*}
and thus, \pref{eq-voinf-somej} holds.
\end{proof}

\begin{proposition}\label{proof:x<->xii}
\pref{eq-srlinf-p-pi-somej} $\Leftrightarrow$ \pref{eq-nlinf-p-pi-somej},
\and
\pref{eq-srlinf-p-pi-allj} $\Leftrightarrow$ \pref{eq-nlinf-p-pi-allj}.
\end{proposition}

\begin{proof}
($\Rightarrow$) Given that $\rho_0 = r(P-\Pi) = \inf\limits_{n\geq 1} || P^n - \Pi ||_{L^{\infty}_{V}}^{1/n} $, for $\rho_{0} < \rho < 1$, there exists $m\in \N$ such that $ || P^{m} - \Pi ||_{L^{\infty}_{V}} < \rho^m<1$. Therefore, for some $m \in \N$,
\begin{equation*}
|| P^{m} - \Pi ||_{L^{\infty}_{V}} < 1.
\end{equation*}

\noindent ($\Leftarrow$) If $|| P^{m} - \Pi ||_{L^{\infty}_{V}} < 1$ for some $m \in \N$,
given that $r(P-\Pi) = \inf\limits_{n\geq 1} \norm{P^n - \Pi}_{L^{\infty}_{V}}^{1/n} $, we have
\begin{equation*}
r(P-\Pi) = \inf\limits_{n\geq 1} \norm{P^n - \Pi}_{\lv}^{1/n} \leq \norm{P^{m} - \Pi}_{\lv}^{1/m} < 1.
\end{equation*}
\end{proof}

\section{Proofs for Reversible Conditions}
\label{sec-reversible}

\begin{proposition}\label{proof:revnewa}
\pref{eq-nl2-p-pi} $\implies$ \pref{eq-l2ge-noC}.
\end{proposition}

\begin{proof}
Let $\rho = \norm{P - \Pi}_{\lpp} < 1$.
Then for each signed measure $\mu \in L^2 (\pi)$,
\[
\norm{\mu (P - \Pi)(\cdot) }_{\lpp} \leq \rho \norm{\mu}_{\lpp}.
\]
Let $\mu \in \lpp$ be a probability measure and let $n\in \N$. By Lemma \ref{lem-P-L2}, $\mu - \mu(\X)\pi = \mu - \pi$ is orthogonal to $\pi$,
so $(\mu - \pi) \Pi = 0$, and hence
$$
\mu P^n - \pi = (\mu-\pi) P^n = (\mu-\pi)(P^n-\Pi)
= (\mu-\pi)(P-\Pi)^n.
$$
Therefore,
\begin{align*}
\norm{\mu P^{n}(\cdot) -\pi(\cdot)}_{\lpp}
&= \norm{(\mu - \pi)( P -\Pi)^{n}(\cdot)}_{\lpp} \\
&\leq \norm{\mu - \pi}_{\lpp} \, \norm{P^{n}- \Pi}_{\lpp} \\
&\leq \norm{\mu - \pi}_{\lpp} \, \rho^{n} .
\qedhere
\end{align*}
\end{proof}

\begin{proposition}\label{proof:xvii->xvi}
\pref{eq-l2ge-noC} $\implies$ \pref{eq-l2ge-Cmu}.
\end{proposition}

\begin{proof}
If \pref{eq-l2ge-noC} holds, for each probability measure $\mu \in \lpp$ and $n\in \N$,
\begin{align*}
    \norm{\mu P^n(\cdot) - \pi(\cdot)}_{\lpp} \leq \norm{\mu - \pi}_{\lpp}\rho^n = C_{\mu}\rho^n,
\end{align*}
with $C_{\mu} = \norm{\mu - \pi}_{\lpp}$.
\end{proof}

\begin{proposition}\label{proof:xvi->xvii'}
\pref{eq-l2ge-Cmu} $\implies$ \pref{eq-gefslp}.
\end{proposition}

\begin{proof}
If \pref{eq-l2ge-Cmu} holds,
then by Lemmas~\ref{lem-TVL1} and~\ref{lem-L1L2},
for each $n\in\N$ and $\mu\in\lpp$ we have
$$
\tvnorm{\mu P^{n}(\cdot) - \pi(\cdot )}
\ = \ \half \,  \norm{\mu P^n(\cdot) - \pi(\cdot)}_{L^1(\pi)}
\ \leq \ \half \, \norm{\mu P^n(\cdot) - \pi(\cdot)}_{L^2(\pi)}
\ \leq \ \half \, C_{\mu} \, \rho^n.
$$
This shows \pref{eq-gefslp} with $p=2$ and $\rho_\mu = \rho$.
\end{proof}

\begin{proposition}\label{proof:iii->xviii}
\pref{eq-gefalp} $\implies$ \pref{eq-nl2-p-0}.
\end{proposition}

\begin{proof}
Take $p=2$ in \pref{eq-gefalp}.
Then it follows from the ``$(iii) \implies (ii)$'' implication
of \cite[Theorem~2]{hybrid} (which is proven by
contradiction, using reversibility and
the spectral measure of $P$ acting on $\lpp$)
that there is $\rho<1$ such that
$$
\|\mu P \|_{\lpp} \ \le \ \rho \, \|\mu\|_{\lpp}
$$
for all probability measures $\mu\in\lpp$ with $\mu(\X)=0$.
Hence, $\norm{P}_{\pi^\bot} \le \rho < 1$.
\end{proof}

\begin{proposition}\label{proof:xxi<->xxii}
\pref{eq-nl2-p-0} $\Leftrightarrow$ \pref{eq-srl2-p-0}.
\end{proposition}

\begin{proof} 
This follows immediately from the fact
(e.g.\ \cite[Proposition VIII.1.11(e)]{conway:2019course})
that, by reversibility,
$r_{\pi^\bot}(P) = ||P||_{\pi^\bot}$.
\end{proof}

\begin{proposition}\label{proof:xx<->xxi}
\pref{eq-nl2-p-0} $\implies$ \pref{eq-nl2-p-pi}.
\end{proposition}

\begin{proof}
From Lemma \ref{lem-equality-l2} it follows that
\[
\norm{P - \Pi}_{\lpp} \ = \ \norm{P}_{\pi^\bot}.
\]
Hence,
if $\norm{P}_{\pi^\bot} < 1$, then $\norm{P - \Pi}_{\lpp}<1$.
\end{proof}

\begin{proposition}\label{proof:xix<->xx}
\pref{eq-srl2-p-pi} $\Leftrightarrow$ \pref{eq-nl2-p-pi}.
\end{proposition}

\begin{proof}
Since $P$ is reversible, $P - \Pi$ is self-adjoint by Lemma~\ref{lem-P-L2}.
Therefore, $r_{\lpp}(P - \Pi) = || P - \Pi ||_{L^{2}(\pi)}$
(e.g.\ \cite[Proposition~VIII.1.11(e)]{conway:2019course}).
Hence,
$||P-\Pi||_{\lpp} < 1$ if and only if $r_{\lpp}(P - \Pi)<1$.
\end{proof}

\begin{proposition}\label{proof:xviii<->xix}
\pref{eq-sg-l2} $\Leftrightarrow$ \pref{eq-srl2-p-0}.
\end{proposition}

\begin{proof}
($\Rightarrow$) If \pref{eq-sg-l2} holds, there is $\rho < 1$ such that  
\begin{equation*}
    \s_{\lpp}(P) \ \subseteq \ \{1\} \cup
 \{ \lambda \in \mathbb{C} : |\lambda| \leq \rho \}.
\end{equation*}
Since 1 is an eigenvalue of multiplicity~1
by Lemma~\ref{lem-P-Linf},
with corresponding eigenvectors
the non-zero constant multiples of $\pi$ which are not in $\pi^\bot$,
we must have
$\s_{\pi^\bot}(P) \subseteq \s_{\lpp}(P) \setminus \{1\}$.  Hence,
$\s_{\pi^\bot}(P) \subseteq \{ \lambda \in \mathbb{C}
: |\lambda| \leq \rho \}$.
Therefore, $r(P|_{\pi^{\bot}}) \le \rho <1$.

($\Leftarrow$) If $r_{\pi^\bot}(P)<1$, there is $\rho<1$ with
$\s_{\pi^\bot}(P)
\subseteq \{ \lambda \in \mathbb{C} : |\lambda| \leq \rho \}$.
So,
by Lemma~\ref{lem-P-L2},
\[
\s_{\lpp}(P)\setminus \{1\} \ \subseteq \ \s_{\pi^\bot}(P)
\ \subseteq \ \{ \lambda \in \mathbb{C} : |\lambda| \leq \rho \}.
\qedhere
\]
\end{proof}

\section{Future Directions and Open Problems}
\label{sec-open}

Our Theorem~\ref{thm-equiv} above provides a fairly complete picture
of equivalences of geometric ergodicity.  However, it does lead to
some additional questions which remain, including:

\renewcommand{\labelenumi}{{\bf Q$\,$\thesection.\arabic{enumi}.} }
\begin{enumerate}

\item
\label{q-irred}
We have assumed throughout that the chain is {\it $\phi$-irreducible}
and {\it aperiodic}.  Those
properties are certainly required for, and implied by, geometric
ergodicity.  But do they need to be assumed explicitly?  Many of our
equivalent conditions {\it imply} them, so that they do
not actually need to be mentioned.  But some of our conditions
do not, e.g.\ the drift
conditions \pref{eq-dc-nomom} and \pref{eq-dc-allj}.  So, which of our
equivalences continue to hold without assuming $\phi$-irreducibility
and aperiodicity?

\item\label{q-countgen}
We also assumed that our state space $(\X,\F)$ is {\it countably
generated}, which holds for e.g.\ the Borel subsets of $\R$ and
of $\R^d$, but {\it not} for e.g.\ the Lebesgue-measurable
subsets.  It is a very standard assumption
(e.g.~\cite[p.~66]{MeynTweedie}), used to ensure the existence of small sets
\cite{Doeblin1940,JainJamison, Orey}
and the measurability of certain functions (e.g.\ \cite[Appendix]{hybrid}).
But which of our equivalences would continue to hold without it?

\item
\label{q-varbound}
The property of {\it aperiodicity} is not necessary for other
important properties such as Central Limit Theorems which involve
averages of functional values like ${1 \over M} \sum_{i=1}^M h(X_i)$.
The weaker notion of {\it variance bounding} essentially
corresponds to geometric ergodicity without aperiodicity, and still
implies CLTs.  Many equivalences to variance bounding
have been proven for {\it reversible}
chains; see~\cite{VarBounding}.  But can equivalences similar to our
Theorem~\ref{thm-equiv} be derived for the variance bounding property
without assuming reversibility?

\item\label{q-reversible}
Our later conditions \condref{eq-l2ge-Cmu} through
\condref{eq-srl2-p-0} were only shown to be equivalent for {\it
reversible} chains.  But are there explicit counter-examples to show that
they are {\it not} equivalent in the absence of reversibility?  Or are
some of them are still equivalent to geometric ergodicity, even without
assuming reversibility?
(For a start on this, \cite[Theorem~1.3]{KM2} proves that
without reversibility the implication
\pref{eq-srl2-p-pi} $\implies$ \pref{eq-geometric-erg}
still holds, but \cite[Theorem~1.4]{KM2}
makes use of \cite{haggstrom} to show that the converse might fail.)

\item Our equivalences are for the fairly strong property of geometric
ergodicity.  But are there similar equivalences for the even
stronger property of {\it uniform ergodicity},
i.e.\ the property that 
$\tvnorm{P^n(x,\cdot) - \pi(\cdot)} \, \le \, C \, \rho^n$
from $\pi$-a.e.\ $x\in\X$ where $C$ does not depend on $x$?
(For a start on this, see \cite[Theorem~16.0.2]{MeynTweedie}.)

\item In the other direction,
are there similar equivalences for the weaker property of
{\it polynomial ergodicity},
i.e.\ the property that 
$\tvnorm{P^n(x,\cdot) - \pi(\cdot)} \, \le \, C_x \, n^{-\alpha}$
for some $\alpha>0$?
(For some discussion and results related to this property,
see e.g.\ \cite{FortMoulines,JarnerPol}.)

\item\label{q-simpleerg}
And, are there similar equivalences for the even weaker property of
{\it simple ergodicity}, i.e.\ the property that just
$\tvnorm{P^n(x,\cdot) - \pi(\cdot)} \, \to \, 0$ as $n\to\infty$
from $\pi$-a.e.\ $x\in\X$, without specifying any rate?
(For a start on this, see e.g.\ \cite[Theorem~13.0.1]{MeynTweedie}.)

\end{enumerate}

\noindent We leave these questions as open problems for future work.

\bigskip\bigskip\noindent\bf Acknowledgements. \rm
We thank Jim Hobert, Galin Jones, and Gareth Roberts for encouraging
us to write this paper, and thank the anonymous referee for a very
careful reading and helpful report.

\bigskip\noindent\bf Note added in proof: \rm
It follows from Proposition~16 on page 3607 of {\it Annals of
Applied Probability} {\bf 25(6)} (2015) that we can also include
the additional equivalent condition:
\begin{itemize}
\item[$vii'$)\ ]
{\sl
There exists a small set $S\in\F$ and constant $\kappa>1$ such
that if $V(x) = \E_x(\kappa^{\tau_S})$ for all $x\in\X$,
then $PV(x) \le \lambda \, V(x) + b \, \one_S(x)$ for all
$x\in\X$, where $\lambda = \kappa^{-1} < 1$ and $b = \sup_{x\in S}
V(x) < \infty$.
}
\end{itemize}

\bigskip\noindent\bf Notes added after publication: \rm
It follows from \cite[Theorem~15.4.1]{MeynTweedie} that
the following condition is implied by $(i)$,
so since it clearly implies $(ii)$ it is also equivalent:
\begin{itemize}
\item[$ii'$)\ ]
{\sl
There exists an absorbing
subset $H\in\F$ with $\pi(H)=1$ such that
there are $\rho<1$ and $C_x<\infty$ such that for all $x\in H$,
            \begin{align*}
                \tvnorm{P^n(x,\cdot) - \pi(\cdot)} \ \leq \ C_x \, \rho^n
\qquad \hbox{\rm for all} \ n\in\N .
            \end{align*}
}
\end{itemize}

\noindent
And, it follows from Theorem~2.1 of Lawler and Sokal
(Trans AMS {\bf 309(2)}, October 1988, 557--580) that for
reversible chains, the spectral gap condition $(xxx)$ is equivalent to:
\begin{itemize}
\item[$xxx'$)\ ]
{\sl
$k>0$, where $k = \inf\limits_{A\in\F \atop 0 < \pi(A) < 1}
{\int_{x\in A} \pi(dx) \, P(x,A^C) \over \pi(A) \, \pi(A^C)}$
is the conductance (Cheeger's constant).
}
\end{itemize}
\bigskip

\ifdec

\section*{Declarations}

This work was supported in part by Discovery Grant RGPIN-2019-04142 from
the Natural Science and Engineering Research Council (NSERC) of Canada.

The authors have no competing interests to declare that are relevant to
the content of this article.

Data sharing is not applicable to this article as no datasets were
generated or analysed during the current study.

\fi

\bibliography{references}

\end{document}